\documentclass[12pt]{amsart}

\font\bbbld=msbm10 scaled\magstephalf

\newcommand{\ba}{\bar{a}}

\newcommand{\bi}{\bar{i}}
\newcommand{\bj}{\bar{j}}
\newcommand{\bk}{\bar{k}}
\newcommand{\bl}{\bar{l}}
\newcommand{\bm}{\bar{m}}
\newcommand{\bn}{\bar{n}}
\newcommand{\bp}{\bar{p}}
\newcommand{\bq}{\bar{q}}

\newcommand{\bs}{\bar{s}}
\newcommand{\bt}{\bar{t}}

\newcommand{\bz}{\bar{z}}

\newcommand{\bM}{\bar{M}}

\newcommand{\bpartial}{\bar{\partial}}

\newcommand{\balpha}{\bar{\alpha}}
\newcommand{\bbeta}{\bar{\beta}}
\newcommand{\eeta}{\bar{\eta}}

\newcommand{\bxi}{\bar{\xi}}
\newcommand{\bzeta}{\bar{\zeta}}

\def \a{\alpha}

\def \p{\partial}
\def \f{\frac}

\def \l{\lambda}

\def \s{\sigma}

\newcommand{\fg}{\mathfrak{g}}

\newcommand{\fRe}{\mathfrak{Re}}

\newcommand{\bfC}{\hbox{\bbbld C}}

\newcommand{\bfR}{\hbox{\bbbld R}}

\newcommand{\cC}{\mathcal{C}}

\newcommand{\cL}{\mathcal{L}}

\newcommand{\cP}{\mathcal{P}}

\newcommand{\Ric}{\mbox{Ric}}
\newcommand{\tr}{\mbox{tr}}

\newcommand{\ol}{\overline}
\newcommand{\ul}{\underline}

\newtheorem{theorem}{Theorem}[section]
\newtheorem{lemma}[theorem]{Lemma}

\newtheorem{corollary}[theorem]{Corollary}

 \theoremstyle{definition}
\newtheorem{definition}[theorem]{Definition}
\newtheorem{example}[theorem]{Example}

\theoremstyle{remark}
\newtheorem{remark}[theorem]{Remark}

\numberwithin{equation}{section}

\begin{document}

\setlength{\baselineskip}{1.2\baselineskip}

\title[Fully Nonlinear Elliptic Equations on Hermitian Manifolds]
{
Second Order Estimates for Fully Nonlinear Elliptic Equations with
Gradient Terms on Hermitian Manifolds}
\author{Bo Guan}
\address{Department of Mathematics, Ohio State University,
         Columbus, OH 43210, USA}
\email{guan@math.ohio-state.edu}
\author{Xiaolan Nie}
\address{College of of Mathematics and Computer Science,
   Zhejiang Normal University,
   Jinhua, Zhejiang Province, 321004 China      }
\email{nie@zjnu.edu.cn} 
\thanks{Research of the first author was supported in part by NSF grants. The second author was partially supported by NSFC (Grant No. 11801516).}

\begin{abstract}
We derive {\em a priori} second order estimates for fully nonlinear elliptic equations which depend on the gradients of solutions in critical ways on Hermitian manifolds. The global estimates we obtained apply to an equation arising from a conjecture by Gauduchon which extends the Calabi conjecture; this was one of the original motivations to this work. We were also motivated by the fact that there had been increasing interests in fully nonlinear pde's from complex geometry in recent years, and aimed to develop general methods to cover as wide a class of equations as possible.

{\em Mathematical Subject Classification (MSC2020):}
35J15, 35J60, 58J05, 53C21, 53C55.

{\em Keywords:} Fully nonlinear elliptic equations, Hermitian manifolds;
 {\em a priori} estimates; concave functions and the associated tangent cones at infinity.
\end{abstract}

\maketitle

\bigskip

\section{Introduction}
\label{gn-I}
\setcounter{equation}{0}
\medskip

Let $(M^n, \omega)$ be a compact Hermitian manifold of
complex dimension $n \geq 2$, and $f$ a symmetric function of $n$ variables.
We consider fully nonlinear elliptic equations of the form
\begin{equation}
\label{3I-10}
f (\lambda (\sqrt{-1} \partial \bpartial u + \chi))= \psi \;\; \mbox{on $M$,}
\end{equation}
where $\chi$ is a real $(1,1)$ form which may depend on $u$ and its gradient,
and $\lambda (X) = (\lambda_1, \cdots, \lambda_n)$ denotes
the eigenvalues of a $(1,1)$ form $X$ with respect to the metric $\omega$.

The function $f$ is assumed to be defined in a symmetric
open and convex cone $\Gamma \subset \bfR^n$ with vertex
at the origin, 
\begin{equation}
\label{3I-15}
\Gamma_n = \{\lambda \in \bfR^n: \lambda_i > 0\} \subset\Gamma,
\end{equation}
and satisfies the following structure conditions
\begin{equation}
\label{3I-20}
f_i \equiv f_{\lambda_i} = \frac{\partial f}{\partial \lambda_i} > 0 \;\;
\mbox{in $\Gamma$}, \;\; 1 \leq i \leq n
\end{equation}
and 
\begin{equation}
\label{3I-30}
\mbox{$f$ is a concave function in $\Gamma$}.
\end{equation}

These conditions, first introduced by Caffarelli-Nirenberg-Spruck~\cite{CNS3}, have become standard in the theory of fully nonlinear PDEs.
According to \cite{CNS3}, equation~\eqref{3I-10} is elliptic by (\ref{3I-20})
for a function $u \in C^{2} (M)$ with
$\chi_u =  \chi + \sqrt{-1} \partial \bpartial u \in \Gamma$;
we call such functions
{\em admissible}.

We assume in addition
\begin{equation}
\label{3I-40}
 \sup_{\partial \Gamma} f < \psi < \sup_{\Gamma} f
\end{equation}
 where
\[ \sup_{\partial \Gamma} f \equiv \sup_{\lambda_0 \in \partial \Gamma}
                   \limsup_{\lambda \rightarrow \lambda_0} f (\lambda) \]
in order for equation~\eqref{3I-10} to remain non-degenerate.

Equation~\eqref{3I-10} covers some of the important elliptic equations in
complex geometry. In particular, it includes the complex
Monge-Amp\`ere equation which has received extensive study from different aspects, going back at least to the work of
Aubin~\cite{Aubin78} and Yau~\cite{Yau78} on compact K\"ahler manifolds in proof of Calabi conjectures, Caffarelli-Kohn-Nirenberg-Spruck~\cite{CKNS} on the Dirichlet problem in $\bfC^n$, and
Bedford-Taylor~\cite{BT76, BT82} on weak solutions and pluripontential theory;
see e.g. \cite{PSS12} for an excellent survey and references. In recent years, there have been increasing interests from complex geometry in more general fully nonlinear elliptic and parabolic equations, including those of form~\eqref{3I-10} in which $\chi$ and $\psi$ may depend on $u$ and/or its gradient. 
Exciting successes have been achieved in the study of some of these equations, such as the Fu-Yau equation~\cite{FY07, FY08} in which $\chi = \chi (z, u)$ and its extensions 
~\cite{PPZ16, PPZ17, PPZ21, CHZ19}, 
and an equation treated by Sz\'ekelyhidi-Tosatti-Weinkove~\cite{STW17} in connection
to a conjecture of Gauduchon which will be discussed in more details in Section~\ref{gn-G}. 
In \cite{PPZ3} and subsequent papers, 
Phong-Picard-Zhang introduced new geometric flows; see also \cite{Phong} and references therein for equations from geometry and physics.
In this paper we wish to treat equations in the general form~\eqref{3I-10} for $\chi = \chi (\cdot, \partial u, \bpartial u)$.

Typical examples of function $f$ satisfying conditions
\eqref{3I-20}-\eqref{3I-30} include $f = \sigma_k^{\frac{1}{k}}$
and more generally
$f = (\sigma_k/\sigma_l)^{\frac{1}{k-l}}$, $0 \leq l <  k \leq n$
($\sigma_0 = 1$) defined on the Garding cone
\[ \Gamma_k = \{\lambda \in \bfR^n: \sigma_j (\lambda) > 0, \;
\mbox{for $1 \leq j \leq k$}\}, \]
as well as $f = \sigma_k/\sigma_{k-1}$ on $\Gamma_{k-1}$ for $1 < k \leq n$,
where
\[ \sigma_k (\lambda) = \sum_{1 \leq i_1 < \cdots < i_k \leq n}
     \lambda_{i_1} \cdots \lambda_{i_k} \] 
is the  $k$-th elementary symmetric function. 
Another important example is 
\[ f (\lambda) = \sum \tan^{-1} \lambda_i \]
which, deeply rooted in calibration geometry of Harvey-Lawson~\cite{HL82}
and theory of special Lagrangian manifolds in the real case,  
corresponds to the deformed Hermitian-Yang-Mill equation; see~\cite{LYZ00} as well as e.g. \cite{JY17, {Collins-Jacob-Yau}, {Collins-Xie-Yau}, {Collins-Yau}, ChenG21, 
HJ21, HZZ21, Lin} for recent related results.

So far most of attentions to equation~\eqref{3I-10} from a more PDE point of view have been on the case where $\chi$ is independent of $u$ and its gradient.
In ~\cite{LiSY04} Li treated the Dirichlet problem in $\bfC^n$.
For $f = \sigma_k^{\frac{1}{k}}$ and $\chi = \omega$ on a compact K\"ahler manifold,
equation~\eqref{3I-10} was first studied by 
Hou-Ma-Wu~\cite{HMW10} who derived second order estimates of the form
\begin{equation}
\label{GN-I100}
\max_M |\partial \bpartial u| \leq C (1 + \max_M |\nabla u|^2),
\end{equation}
followed by Dinew-Kolodziej~\cite{DK} who used a blow-up argument and Liouville type theorem to derive a gradient bound from \eqref{GN-I100}.
In ~\cite{SW08},  Song-Weinkove introduced a necessary and sufficient cone condition to solve equation~\eqref{3I-10} for
$f = \sigma_n/\sigma_{n-1}$, known as J-equation, in answer to a question raised by Donaldson~\cite{Donaldson99a} in the K\"ahler case where $\chi$ is another K\"ahler metric on $M$ and $\psi$ is an invariant constant given by
\[ \psi = \frac{\int_M \chi^n}{\int_M \chi^{n-1} \wedge \omega}; \]
see also \cite{Weinkove04, Weinkove06}. Their results were extended to $f = (\sigma_n/\sigma_l)^{\frac{1}{n-l}}$,
$1 \leq l < n$ by Fang-Lai-Ma~\cite{FLM11} 
and by Sun~\cite{Sun16, Sun17} who treated general $\psi$.
From a point of view in analogue to Yau-Tian-Donaldson stability~\cite{Yau, Tian97, Donaldson02},
Lejmi-Sz\'ekelyhidi~\cite{LS15} proposed as a conjecture an integral condition for J-equation and solved it when $n = 2$. 
In higher dimensions, the conjecture was proved for toric manifolds by
Collins-Sz\'ekelyhidi~\cite{CS17} who also treated Krylov type equations, extending 
 results of ~\cite{FLM11} and Zheng~\cite{ZhengK15}, and by Song~\cite{SongJ} in the general case; see also Chen~\cite{ChenG21} and Datar-Pingali~\cite{DP}.

For equation~\eqref{3I-10} on compact Hermitian manifolds,
Zhang~\cite{ZhangDK17} and Sun~\cite{Sun17b} solved the cases
$f = \sigma_k^{\frac{1}{k}}$ and $f =  (\sigma_k/\sigma_l)^{\frac{1}{k-l}}$ respectively, and Sz\'ekelyhidi~\cite{Szekelyhidi} gave a thorough treatment
 for general $f$ and  $\psi = \psi (z)$, under the assumption
 in addition to \eqref{3I-20}-\eqref{3I-40},
\begin{equation}
\label{GN-I105}
\sum f_i \lambda_i  \geq 0, \;\; \sum f_i \geq c_0 > 0.
\end{equation}

A crucial ingredient in solving equation~\eqref{3I-10} is to establish {\em a priori}  bounds for second derivatives of admissible solutions, on which we shall focus 
in this paper. 
Our main results will be stated in Section~\ref{gn-R}; see Theorems~\ref{gj-th4}, 
\ref{3I-th4} and \ref{3I-th40}. We obtain in Theorems~\ref{gj-th4} an interior estimates for second derivatives 
under an assumption on $\chi$ (see \eqref{A3})
analogous to Ma-Trudinger-Wang condition~\cite{MTW05} in optimal transport theory.
Theorems~\ref{3I-th4} and \ref{3I-th40} concern global second order estimates in which
the primary assumptions are in terms of the tangent cone at infinity introduced in 
\cite{Guan14} for the level hypersurface $\{f = \psi\}$ which is smooth and convex by conditions \eqref{3I-20} and \eqref{3I-30}. 
In the case $\chi = \chi (z)$, i.e. when it is independent of $\partial u$ and $\bpartial u$, Theorem~\ref{3I-th4} recovers the second order estimates of Sz\'ekelyhidi~\cite{Szekelyhidi} without assumption~\eqref{GN-I105} for type I cones
(as defined in Caffarelli-Nirenberg-Spruck~\cite{CNS3}); for details see Section~~\ref{gn-R}.

In what follows we describe an application of Theorem~\ref{3I-th4}
which was one of the original motivations to our work in the current article.

Let $\omega_0$ be another Hermitian metric on $M^n$.  We consider the
 Monge-Amp\`ere form-type equation
\begin{equation}
\label{CH-I10}
\det \Phi_u = e^{h + b} \det (\omega^{n-1})
\;\; \mbox{in $M$}
\end{equation}
with
\begin{equation}
\label{CH-I15}
\Phi_u = \omega_0^{n-1}  + \sqrt{-1}\partial \bpartial u \wedge \omega^{n-2}
    +  c \, \fRe \{\sqrt{-1}\partial u \wedge \bpartial \omega^{n-2}\} > 0
\end{equation}
where $h$ is a given function, and $b, c$ are constant.

Equation~\eqref{CH-I10} is related to the general notion of plurisubharmonic functions of Harvey-Lawson~\cite{HL11, HL12}. 
It was studied by Fu-Wang-Wu~\cite{FWW11, FWW15},
Tosatti-Weinkove~\cite{TWv13a, TWv13b}  for $c = 0$,
and by Sz\'ekelyhidi-Tosatti-Weinkove~\cite{STW17} for $c = 1$, proving
a conjecture of Gauduchon~\cite{Gauduchon84} (see also Conjecture 1.3 in \cite{TWv13b}).
As a consequence of Theorem~\ref{3I-th4} we obtain

\begin{corollary}
\label{GN-cor-I10}
Let $u \in C^4 (M)$ be a solution of equation~\eqref{CH-I10} satisfying \eqref{CH-I15}.
Then \eqref{GN-I100} holds for some constant $C$ depending on $|u|_{C^0 (M)}$ and known data.
\end{corollary}

It was proved by Tosatti-Weinkove~\cite{TWv13b} that the solvability of
equation~\eqref{CH-I10} reduces to establishing \eqref{GN-I100}; see also
Sz\'ekelyhidi~\cite{Szekelyhidi}. In their joint paper \cite{STW17}, these authors
carried out such an estimate.

Among other examples of $f$ satisfying \eqref{3I-20}-\eqref{3I-30}
are $f = \log \rho_k$, $1 \leq k \leq n$, where
\[ \rho_k (\lambda) := \prod_{1 \leq i_1 < \cdots < i_k \leq n}
(\lambda_{i_1} + \cdots + \lambda_{i_k})  \]
defined in the cone
\[ \mathcal{P}_k : = \{\lambda \in \bfR^n:
      \lambda_{i_1} + \cdots + \lambda_{i_k} > 0,
         \; \forall \; 1 \leq i_1 < \cdots < i_k \leq n \}. \]
In particular,  $\rho_1 = \sigma_n$ and $\rho_n = \sigma_1$.

We note that in the literature of nonlinear PDE, a domain $\Omega \in \bfR^n$ or its boundary
$\partial \Omega$ is often called
$k$-convex or $k$-mean convex if $\kappa \in \Gamma_k$ where
$\kappa = (\kappa_1, \ldots, \kappa_{n-1})$ denotes the principal curvatures of
$\partial \Omega$, while in areas of geometry/geometric analysis such as mean curvature flows, that a hypersurface $\Sigma$ in $\bfR^n$ is $k$-convex sometimes means that its principal curvatures $\kappa \in \mathcal{P}_k$. The latter concept was first introduced  in 1986 by Sha~\cite{Sha86}. In \cite{HS09} and \cite{BH17}, Brandle, Huisken and Sinestrari studied regularity of mean curvature flow of two-convex hypersurfaces in $\bfR^n$.

Using the Hodge star operator, Tosatti-Weinkove~\cite{TWv13a} were able to rewrite
equation~\eqref{CH-I10} in the form of \eqref{3I-10}
for $f = \log \rho_{n-1}$. We shall derive Corollary~\ref{GN-cor-I10}
by verifying the conditions in Theorem~\ref{3I-th4} for this equation.

In~\cite{PPZ16b}, Phong-Picard-Zhang 
considered equation~\eqref{3I-10} in which $\psi = \psi (z, u, \nabla u)$ for 
$f = \sigma_k^{\frac{1}{k}}$ and $\chi = \chi (z, u) \geq c_0 \omega > 0$ on a compact K\"ahler manifold $(M^n, \omega)$, and derived second order estimates which allow
arbitrary dependence of $\psi$ on $\nabla u$, in the spirit of 
 Guan-Ren-Wang~\cite{GRW15}. We obtain similar results under assumption~\eqref{A3}; see Remarks~\ref{gn-remark-I10} and \ref{gn-remark-I20}.

The Dirichlet problem for equation~\eqref{3I-10} on Hermitian manifolds was treated for
$f = (\sigma_n/\sigma_l)^{\frac{1}{n-l}}$, $1 \leq l < n$ by 
Li, Sun and the first author~\cite{GL13, GS15} where gradient estimates were derived directly using the maximum principle method, and more recently by Collins-Picard~\cite{CP} for 
$f = \sigma_k^{\frac{1}{k}}$, $1 < k < n$ and Feng-Ge-Zheng~\cite{FHZ} for Hessian quotient equations. See also \cite{GQY19, QY17, Yuan18} where equations with gradient dependence were studied.
For general $f$ it seems still open, while the corresponding Dirichlet problem on Riemannian manifolds was studied in \cite{Guan} for $\chi = \chi (z)$, $\psi = \psi (z)$ and \cite{GJ15, GJ16} in the more general case $\chi = \chi (z, u, \nabla u)$, 
$\psi = \psi (z, u, \nabla u)$ .

The article is organized as follows.
Section~\ref{gn-R} contains the statements of our main theorems. 
In Section~\ref{gn-P} we prove a key inequality needed in the proof of 
Theorem~\ref{3I-th40} for $f$, extending a result in \cite{Guan14}.  
Section~\ref{gn-C2}  is devoted to second derivative estimates, completing the proof of the main results while in Section~\ref{gn-G} we prove 
Corollary~\ref{GN-cor-I10}.

Part of research described in this paper was done while the second author was a Ross Assistant Professor in Department of Mathematics at The Ohio State University. 
More recently we were able to refine some of the calculations, so hopefully it was more accessible to the reader and especially those new to the field. We also up-datetd the references, adding papers appeared more recently in the field.
We wish to thank Gabor Sz\'ekelyhidi, Valentino Tosatti and Ben Weinkove for communications on their results in \cite{STW17}.

\bigskip

\section{Notations and Main Results}
\label{gn-R}
\setcounter{equation}{0}
\medskip

Throughout the paper we write in local coordinates $(z_1, \ldots, z_n)$
\[ \omega = \sqrt{-1} g_{i\bj} \, dz_i \wedge d \bz_j. \]
Let $\{g^{i\bj}\} = \{g_{i\bj}\}^{-1}$ denote the inverse matrix of 
 $\{g_{i\bj}\} > 0$.

For fixed $z \in M$, $h \in \bfR$ and $(1, 0)$ form $\zeta$, 
we use the notation
\[ \chi (z, \zeta,  \bzeta, h) = \sqrt{-1} \chi_{i\bj} (z, \zeta,  \bzeta, h) dz_i \wedge d \bz_j \]
or simply $\chi = \sqrt{-1} \chi_{i\bj} dz_i \wedge d \bz_j$,
and
\[ \chi_{\xi \eeta} = \sum \chi_{i\bj} \xi_i \bar \eta_j \]
for
\[ \xi = \sum \xi_i \frac{\partial}{\partial z_i}, \;
    \eta = \sum \eta_i \frac{\partial}{\partial z_i}. \] 


For a function $u \in C^{2} (M)$, we write
$\chi [u] = \chi (\cdot, \nabla u, u)$ or
$\chi [u] = \chi (\cdot, \partial u, \bpartial u, u)$
to indicate the dependence of $\chi$ on $u$ and its gradient,
and 
\[ \fg_{i\bj} =  \chi_{i\bj} [u] + \nabla_{\bj} \nabla_i u
= \chi_{i\bj} [u] + \partial_{\bj} \partial_i u \]
so
\[ \begin{aligned}
  \chi_u := \chi [u] + \sqrt{-1} \partial \bpartial u = \,& \sqrt{-1} \fg_{i\bj} \, dz_i \wedge d \bz_j.
\end{aligned} \]

In the current paper we shall consider the case $\psi = \psi (z)$, $\chi [u] = \chi (z, \partial u, \bpartial u)$,
and use the following expressions to distinguish different derivatives
 \[   \chi_{i\bj ,k} = \frac{\p \chi_{i\bj}}{\p z_k}, \;
 \chi_{i\bj k} = \chi_{i\bj, k} - \Gamma_{ki}^p \chi_{p\bj}, \;
   \chi_{i\bj k \bl} = \chi_{i\bj k, \bl} - \ol{\Gamma_{lj}^q} \chi_{i\bq k}, \]
\[      \chi_{i\bj, \zeta_{\a}} = \frac{\p \chi_{i\bj}}{\p\zeta_{\a}}, \;
   \chi_{i\bj, \zeta_{\a} k} = \frac{\p^2\chi_{i\bj}}{\p z_k \p \zeta_{\a}},  \;
       \chi_{i\bj k,  \zeta_{\a}} = \frac{\p \chi_{i\bj k}}{\p \zeta_{\a}},  \]
as well as
\begin{equation}
\label{gn-R50}
\partial_k {\chi}_{i\bj} [u]
    = \chi_{i\bj, k}  [u] + \chi_{i\bj, \zeta_{\alpha}}  [u] \p_k \p_{\alpha } u
    + \chi_{i\bj, \bzeta_{\alpha}}  [u] \p_k \p_{\balpha} u
\end{equation}
and
\begin{equation}
\label{gn-R60}
\begin{aligned}
   \nabla_k {\chi}_{i\bj}  [u]
   = \,& \partial_k {\chi}_{i\bj}  [u] - \Gamma_{ki}^p \chi_{p\bj}  [u]
 = \chi_{i\bj k}  [u] + \chi_{i\bj, \zeta_{\alpha}}  [u] \p_k \p_{\alpha } u
    + \chi_{i\bj, \bzeta_{\alpha}}  [u] \p_k \p_{\balpha} u, 
     \end{aligned}
\end{equation}
etc;
we shall drop $[u]$ in the expressions when no confusions would arise.
Similarly,
\begin{equation}
\label{gn-R70}
\nabla_{\bl} \chi_{i\bj k}
    = \chi_{i\bj k \bl}  + \chi_{i\bj k, \zeta_{\alpha}}   \p_{\bl} \p_{\alpha} u
       + \chi_{i\bj k, \bzeta_{\alpha}}  \p_{\bl} \p_{\balpha } u
\end{equation}
and
\begin{equation}
\label{gn-R80}
\begin{aligned}
   \nabla_{\bl}  \chi_{i\bj, \zeta_{\alpha}}
   = \,& \partial_{\bl} \chi_{i\bj, \zeta_{\alpha}}
      + \ol{\Gamma}_{lj}^q \chi_{i\bq, \zeta_{\alpha}} \\
 =  \,& \chi_{i\bj, \zeta_{\alpha} \bl}
 + \chi_{i\bj, \zeta_{\alpha} \zeta_{\beta}} \p_{\bl} \p_{\beta} u
   + \chi_{i\bj, \zeta_{\alpha} \zeta_{\bbeta}} \p_{\bl} \p_{\bbeta} u
   + \ol{\Gamma}_{lj}^q \chi_{i\bq, \zeta_{\alpha}}\\
    =  \,& \chi_{i\bj, \bl \zeta_{\alpha} }
 + \chi_{i\bj, \zeta_{\alpha} \zeta_{\beta}} \p_{\bl} \p_{\beta} u
   + \chi_{i\bj, \zeta_{\alpha} \zeta_{\bbeta}} \p_{\bl} \p_{\bbeta} u
   + \ol{\Gamma}_{lj}^q \chi_{i\bq, \zeta_{\alpha}}\\
   =  \,& \chi_{i\bj\bl, \zeta_{\alpha} }
 + \chi_{i\bj, \zeta_{\alpha} \zeta_{\beta}} \p_{\bl} \p_{\beta} u
   + \chi_{i\bj, \zeta_{\alpha} \zeta_{\bbeta}} \p_{\bl} \p_{\bbeta} u.
     \end{aligned}
\end{equation}

\begin{example}
We explain the above notation by a simple example. Let 
\[ \chi [u] = \partial u \wedge \bpartial u = \partial_i u \bpartial_j u \, dz_i \wedge d\bz_j. \]
Then 
\[   \chi_{i\bj ,k} = 0, \;\; \chi_{i\bj k} = - \Gamma_{ki}^p \partial_p u \bpartial_j u, \;\;
     \chi_{i\bj, \zeta_{\a}} = \delta_{i \alpha} \bpartial_j u, \;\; 
         \chi_{i\bj, \bzeta_{\a}} = \delta_{j \alpha} \partial_i u,  \]
and 
\[ \nabla_k {\chi}_{i\bj}  [u] = - \Gamma_{ki}^p \partial_p u \bpartial_j u
+ \p_k \p_i u \bpartial_j u + \partial_i u \p_k \bpartial_{j} u, 
    =  \nabla_k \nabla_i u \nabla_{\bj} u +\partial_i u \p_k \bpartial_{j} u. \]
\end{example}

In our first result we establish an interior estimate for second derivatives,
which requires the additional assumption that there exists $c_0 > 0$ such that
\begin{equation}
\label{A3}
\sum  \chi_{i\bj, \zeta_k \bar{\zeta}_l} (z, \cdot, \cdot)  
    \xi_i \xi_{\bj} \eta_k \bar{\eta}_l
\leq - c_0 |\xi|^{2} |\eta|^{2}, \;
\forall \, \xi, \eta \in T_z^{1,0} M, \; \omega (\xi, \bar\eta) = 0.
\end{equation}
This is an analogue of assumption (A3) of Ma-Trudinger-Wang~\cite{MTW05}.

\begin{theorem}
\label{gj-th4}
Let $u \in C^{4, \alpha} (B_{R})$ be an admissible solution of equation~\eqref{3I-10} in a geodesic ball $B_{R} \subset M$ of radius $R$, where $0 < \alpha < 1$.
Assume $\psi = \psi (z)$ and that \eqref{3I-20}-\eqref{3I-40}, 
\eqref{A3} holds. Then $u$ satisfies
the interior a priori estimate
\begin{equation}
\label{gj-G185}
|\nabla^2 u|_{C^{\alpha} (B_{R/2})}  \leq C
\end{equation}
where $C$ depends on $c_0$, $R^{-1}$, $|\nabla u|_{C^1 (\ol{B}_{3R/4})}$ and
\[ 
   c_1 = \sup_{\Gamma} f - \sup_{B_{R}} \psi > 0. \]
\end{theorem}

\begin{remark}
\label{gn-remark-I10}
Theorem~\ref{gj-th4} still holds for $\psi = \psi (z, u, \nabla u)$ provided in addition
that either $\psi (z, u, \nabla u)$ is convex in $\nabla u$ or
\begin{equation}
\label{gn-I185}
\lim_{|\lambda| \rightarrow + \infty} \sum f_i = + \infty, 
\;\; \forall \; \lambda \in \Gamma.
 \end{equation}
 This holds for $f = \sigma_k^{\frac{1}{k}}$ on $\Gamma_k$, $k > 1$
 \end{remark}

It is well known that in general there are no interior second order estimates for fully nonlinear elliptic equations. In particular, Theorem~\ref{gj-th4} fails without condition~\eqref{A3}, even in the case $\chi = \chi (z)$.

Turning to global second derivative estimates where assumption~\eqref{A3} is dropped,
we first recall some notions from \cite{Guan14}.
For $\sigma \in (\sup_{\partial \Gamma} f, \sup_{\Gamma} f)$
define
\[ \Gamma^{\sigma} = \{\lambda \in \Gamma: f (\lambda) > \sigma\}. \]
By \eqref{3I-20} and \eqref{3I-30},
$\partial \Gamma^{\sigma} = \{\lambda \in \Gamma: f (\lambda) = \sigma\}$ is a smooth and convex complete hypersurface in $\Gamma$.
For $\lambda \in \partial \Gamma^{\sigma}$ let
$\nu_{\lambda} = Df (\lambda)/|Df (\lambda)|$
denote the unit normal vector to $\partial \Gamma^{\sigma}$ at $\lambda$.

\begin{definition}[\cite{Guan14}] 
For $\mu \in \bfR^n$
let 
 \[ S^{\sigma}_{\mu} = \{\lambda \in \partial
\Gamma^{\sigma}: \nu_{\lambda} \cdot (\mu - \lambda) \leq 0\}. \]
The  {\em tangent cone at infinity} 
to $\Gamma^{\sigma}$ is defined as
\[ \begin{aligned}
\cC^+_{\sigma}
 \,& = \{\mu \in \bfR^n:
              S^{\sigma}_{\mu} \; \mbox{is compact}\}.
    \end{aligned} \]
\end{definition}

Clearly, $\cC^+_{\sigma}$ is a symmetric convex cone, and 
$\Gamma^{\sigma} \subset \cC^+_{\sigma}$.
The following results were proved in \cite{Guan14}.

\begin{theorem}[\cite{Guan14}]
\label{3I-th4.5}
{\bf a)}
$\mathcal{C}_{\sigma}^+$ is open.
{\bf b)}
Let $\mu \in {\mathcal{C}}_{\sigma}^+$.
There exist $\varepsilon, R > 0$ such that
\begin{equation}
\label{gn-I115}
f_i (\lambda) (\mu_i - \lambda_i) \geq \varepsilon \sum f_i (\lambda) + \varepsilon,
\;\; \forall \; \lambda \in \partial \Gamma^{\sigma} \setminus B_R (0).
\end{equation}
\end{theorem}

We next introduce a new quantity which will also play a key role in a forthcoming paper~\cite{GGQ}.

\begin{definition}
The {\em rank of ${\mathcal{C}}_{\sigma}^+$} is defined to be
\[ \min \{r (\nu): \mbox{$\nu$ is the unit normal vector of a supporting plane
to ${\mathcal{C}}_{\sigma}^+$}\}. \]
where for a unit vector
$\nu \in \ol{\Gamma}_n$, $r (\nu)$ denotes the number of non-zero
components of $\nu$.
For convenience we define the rank of $ \bfR^n$ to be $n$.
\end{definition}

\begin{remark}
For $f = \sigma_k^{\frac{1}{k}}$ defined on $\Gamma_k$, the rank of
$
\mathcal{C}_{\sigma}^+$ is $n-k+1$.
This follows from an inequality of Lin-Trudinger~\cite{LT1994}.
\end{remark}

It is also easy to see 

\begin{lemma}
\label{gn-lemma-R20}
For $f = \log \rho_k$ defined on $\cP_k$, $1 \leq k \leq n$, the rank of
$
\mathcal{C}_{\sigma}^+$ is $k$.
\end{lemma}

Our second main result may be stated as follows.

\begin{theorem}
\label{3I-th4}
Let $u \in C^4 (M)$ be an admissible solution of \eqref{3I-10} with $\psi \in C^2 (M)$ and let
\[ r_0 = \min \,
   \big\{\mbox{rank of ${\mathcal{C}}_{\psi (z)}^+$}: z \in M\big\}. \]
Assume in addition to \eqref{3I-20}-\eqref{3I-40} that 
\begin{equation}
\label{A2}
\mbox{$\chi_{\xi \bxi} (\cdot, p)$ is a concave function in $p \in T^*_z (M)$},
\; \forall \;  \xi \in T_z^{1,0} M
\end{equation}
where $T^*_z (M)$ denotes the real cotangent space of $M$ at $z$, and that
at any point on $M$ where in local coordinates
$g_{i\bj} = \delta_{ij}$ and $\fg_{i\bj} = \delta_{ij} \lambda_i$ with
$\lambda_1 \geq \cdots \geq \lambda_n$,
\begin{equation}
\label{A5}
\Big|\sum f_i \chi_{i \bi \bar{1}, \zeta_{\alpha}}\Big|
 + \sum f_i |\chi_{i \bar{1}, \zeta_{\alpha}}|^2
   \leq C\lambda_1f_{\alpha},
  \;\; \forall \,  \alpha \leq n - r_0.
\end{equation}
 Then the following estimate holds
\begin{equation}
\label{hess-a10c}
\max_M |\partial \bpartial u| \leq C_1 e^{C_2 (u - \inf_M u)}
\end{equation}
where $C_1$ depends on $|\nabla u|_{C^0 (M)}$ and $C_2$ is a
uniform constant, 
provided that there exists a function $\ul{u} \in C^2 (M)$ satisfying
\begin{equation}
\label{3I-200}
\lambda (\chi_{\ul{u}} (z)) \in {\mathcal{C}}_{\psi (z)}^+
\;\; \forall \; z \in M.
\end{equation}

\end{theorem}


\begin{remark}
For $\chi = \chi (z)$ we obtain \eqref{hess-a10c} under conditions~\eqref{3I-20}-\eqref{3I-40} and \eqref{3I-200}. 
\end{remark}

\begin{remark}
In \cite{Szekelyhidi}, Sz\'ekelyhidi introduced the notation of 
$\cC$-subsolution which has been used widely.
A function $\ul{u} \in C^2 (M)$ is called a $\cC$-subsolution if
\begin{equation}
\label{3I-200S}
(\lambda (\chi_{\ul{u}} (z)) + \Gamma_n) \cap \partial \Gamma^{\psi (z)}
\; \mbox{is compact},
\;\; \forall \; z \in M.
\end{equation}
Among other interesting results, he derived the second order estimate assuming 
~\eqref{3I-20}-\eqref{3I-40}, \eqref{GN-I105} and the existence of a $\cC$-subsolution.
\end{remark}

\begin{remark}
It was shown in \cite{Guan} that if $\Gamma$ is a type I cone, then conditions 
\eqref{3I-200} and \eqref{3I-200S} are equivalent. A  cone $\Gamma$ is of type I if 
each positive $\lambda_i$-axis belongs to the boundary of $\Gamma$; see 
Caffarelli-Nirenberg-Spruck~\cite{CNS3} for definition. For $k > 1$, $\Gamma_k$ is a type I cone while $\Gamma_1$ is a type II cone, meaning not of type I.
\end{remark}

\begin{remark}
For $f = (\sigma_n/\sigma_l)^{\frac{1}{n-l}}$,
$1 \leq l < n$, conditions 
\eqref{3I-200} and \eqref{3I-200S} are all equivalent to the cone condition of
Song-Weinkove~\cite{SW08} and Fang-Lai-Ma~\cite{FLM11}; see e.g. \cite{Szekelyhidi}.
\end{remark}

We now introduce a larger cone containing $\mathcal{C}_{\sigma}^+$.
Note that the unit normal vector of any supporting hyperplane to
$\mathcal{C}_{\sigma}^+$ belongs to $\ol{\Gamma_n}$.
We define $\widetilde{\mathcal{C}}_{\sigma}^+$  to be the region in $\bfR^n$
bounded by those supporting hyperplanes to $\mathcal{C}_{\sigma}^+$ with
unit normal vector in $\partial \Gamma_n$;
so $\widetilde{\mathcal{C}}_{\sigma}^+ = \bfR^n$ if there are no such
supporting planes.

Clearly, the rank of $\widetilde{\mathcal{C}}_{\sigma}^+$ is equal to that of
${\mathcal{C}}_{\sigma}^+$.
Moreover, $\mu + \Gamma_n \subset \widetilde{\mathcal{C}}_{\sigma}^+$
for $\mu \in \widetilde{\mathcal{C}}_{\sigma}^+$
and  $\widetilde{\mathcal{C}}_{\sigma}^+ \subset \widetilde{\mathcal{C}}_{\rho}^+$ if
$\sigma \geq \rho$.

We have the following extension of Theorem~\ref{3I-th4}. 

\begin{theorem}
\label{3I-th40}
Suppose that $f$ satisfies the additional assumption
\begin{equation}
\label{3I-180}
\sum f_i \geq c_0 > 0 \;\; \mbox{in $\{\lambda \in \Gamma: \inf_M \psi \leq f (\lambda) \leq \sup_M \psi\}$}.
\end{equation}
Theorem~\ref{3I-th4} then still holds with assumption \eqref{3I-200} replaced by
\begin{equation}
\label{3I-200'}
\lambda (\chi_{\ul{u}} (z)) \in \widetilde{\mathcal{C}}_{\psi (z)}^+
\;\; \forall \; z \in M.
\end{equation}

\end{theorem}


\bigskip

\section{The cone $\widetilde{\mathcal{C}}_{\sigma}^+$ and extension of Theorem~\ref{3I-th4.5}}
\label{gn-P}
\setcounter{equation}{0}

\medskip

A crucial tool in the proof of Theorem~\ref{3I-th4} is
Theorem~\ref{3I-th4.5}.
To prove Theorem~\ref{3I-th40}, we need the following extension to
$\widetilde{\mathcal{C}}_{\sigma}^+$ in a slightly weaker form.

\begin{theorem}
\label{gn-th10}
Let $\mu \in \widetilde{\mathcal{C}}_{\sigma}^+$.
There exist $\delta, \varepsilon > 0$ such that for any
$\lambda \in \partial \Gamma^{\sigma}$, 
either \begin{equation}
\label{gn-I100}
f_k (\lambda) \geq \delta \sum f_i (\lambda), \;\; \forall \, k
\end{equation}
 or
\begin{equation}
\label{gn-I110}
\sum f_i (\lambda) (\mu_i - \lambda_i) \geq
  \varepsilon \sum f_i (\lambda). 
\end{equation}
\end{theorem}


\begin{proof} 
We prove by contradiction. Assume Theorem~\ref{gn-th10} is false. There exists
$\tilde{\mu} \in \widetilde{\mathcal{C}}_{\sigma}^+$ and
$\lambda_k \in \partial \Gamma^{\sigma}$ for each positive integer $k$
such that
 \begin{equation}
\label{gn-P10}
   \sum f_i (\lambda_k) (\tilde{\mu}_i - (\lambda_k)_i) \leq
  \dfrac{1}{k \sqrt{n}} \sum f_j (\lambda_k) \leq \dfrac{1}{k} |\nabla f(\lambda_k)|,
  \end{equation}
\begin{equation}
\label{gn-P15}
\min_i \{f_i (\lambda_k)\} \leq
    \dfrac{1}{k \sqrt{n}} \sum f_j (\lambda_k) \leq \dfrac{1}{k} |\nabla f(\lambda_k)|
\end{equation}
and therefore
\begin{equation}
\label{gn-P20}
\mbox{dist} (\nu_{\lambda_k}, \partial \Gamma_n)
= \min_i\dfrac{f_i(\l_k)}{|\nabla f(\l_k)|}\leq \frac{1}{k}.
\end{equation}

By the concavity of $f$, for any ${\mu} \in \Gamma^{\sigma}$ we have
\begin{equation}
\label{gn-P30}
0 <  f ({\mu}) - f(\lambda_k)
   \leq \sum f_i (\lambda_k) ({\mu}_i - (\lambda_k)_i).
\end{equation}
 Thus $\nu_{\lambda_k} \cdot \lambda_k \leq |{\mu}|$
(and is therefore uniformly bounded above).

On the other hand, from \eqref{gn-P10} we see that
\[  \nu_{\lambda_k} \cdot \lambda_k \geq  \nu_{\lambda_k} \cdot \tilde{\mu} - \frac{1}{k}
 \geq - |\tilde{\mu}| - 1.  \]
 This shows that $\{\nu_{\lambda_k} \cdot \lambda_k\}$ is bounded from below.

  Consequently,
 passing to a subsequence we may assume
\begin{equation}
\label{gn-P40}
\lim_{k \rightarrow \infty} \nu_{\lambda_k} = \nu 
\end{equation}
 and
\begin{equation}
\label{gn-P50}
 \lim_{k \rightarrow \infty} \nu_{\lambda_k} \cdot \lambda_k = c.
 \end{equation}
We have $\nu \in \partial \Gamma_n$ by \eqref{gn-P20},
and it follows from \eqref{gn-P30} and \eqref{gn-P50} that $\Gamma^{\sigma}$ is contained in the half space
\[ H_{\nu}^+ = \{\mu \in \bfR^n: \nu \cdot \mu > c\}. \]
As $\nu \in \partial \Gamma_n$, this implies $\mathcal{C}^+_{\s} \subset H_{\nu}^+$
and therefore 
$\widetilde{\mathcal{C}}^+_{\s} \subset H_{\nu}^+$.
Moreover, it follows from \eqref{gn-P10}  that $\nu \cdot \tilde{\mu} = c$
showing $\tilde{\mu} \in  \partial H_{\nu}^+$. Consequently, 
$\tilde{\mu} \in  \partial \widetilde{\mathcal{C}}^+_{\sigma}$ which is a
contradiction.
\end{proof}

It would be desirable to improve \eqref{gn-I110} to \eqref{gn-I115}; whether this is possible, however, is not clear to us at the moment.
For fixed $\mu \in \widetilde{\mathcal{C}}^+_{\sigma}$ let $\varepsilon > 0$ satisfy
\eqref{gn-I110} in Theorem~\ref{gn-th10} and denote
\[ A = \{\lambda \in \partial \Gamma^{\sigma}: \eqref{gn-I110} \; \mbox{holds}\}. \]
For $\lambda \in A$ let
\[ t_{\lambda} = \min \{t \geq 0: t \lambda + (1-t) \mu \in \ol{\Gamma^{\sigma}}\}. \]
We see that $t_{\lambda} < 1$, for otherwise \eqref{gn-I110} would not hold.
 As in \cite{Guan14} by the concavity of $f$,
 \[ \sum f_i (\lambda) (\mu_i - \lambda_i) \geq
     \sup_{t_{\lambda} \leq t \leq 1} f (t \lambda + (1-t) \mu) - \sigma > 0 \]
for $\lambda \in A$, as otherwise $t \lambda + (1-t) \mu \in \partial {\Gamma^{\sigma}}$
for all $t_{\lambda} \leq t \leq 1$ which contradicts \eqref{gn-I110}.
Clearly for $r$ large,
\[ h_{\mu} (r) :=
  \min_{\lambda  \in A \cap \partial B_r (0)} 
  \sup_{t_{\lambda} \leq t \leq 1} f (t \lambda + (1-t) \mu) - \sigma > 0\]
since $A \cap \partial B_r (0)$ is compact.
We end this section with the following question:
Is $h_{\mu} (r)$ nondecreasing in $r$?
A positive answer to this question would give an improvement of
Theorem~\ref{gn-th10} which enables us to drop assumption~\eqref{3I-180}
in Theorem~\ref{3I-th40}.

\bigskip

\section{The second order estimates}
\label{gn-C2}
\setcounter{equation}{0}
\medskip

In this section we derive the second order estimates in Theorems~\ref{gj-th4}
and \ref{3I-th4}.
Throughout the section, we use $\nabla$ to denote the Chern connection of $(M, \omega)$ and let $u \in C^4 (M)$ be an admissible solution of equation~\eqref{3I-10}.

In local coordinates $z = (z_1, \ldots, z_n)$, 
equation~\eqref{3I-10} can be written in the form
\begin{equation}
\label{cma2-M10}
 F (\fg_{i\bj}) = \psi
\end{equation}
where the function $F$ is defined by $F (X) = f (\lambda (X))$ for a real
$(1,1)$ form $X$ on $M$.
As usual we denote
\[ F^{i\bj} = \frac{\partial F}{\partial \fg_{i\bj}}, \;
   F^{i\bj, k\bl} = \frac{\partial^2 F}{\partial \fg_{k\bl} \partial \fg_{i\bj}}. \]

We use an idea of Tosatti-Weinkove~\cite{TWv13a}
and consider the quantity
which is given in local coordinates
\begin{equation}
\label{gblq-C10A}
 A := \sup_{z \in M} \max_{\xi \in T^{1,0}_z M}
  \; e^{(1 + \gamma) \phi} \fg_{p\bq} \xi_p \bar{\xi_q}
   (g^{k\bl} \fg_{i\bl} \fg_{k\bj} \xi_i \bar{\xi_j})^{\frac{\gamma}{2}}/|\xi|^{2 + \gamma}
\end{equation}
   where $\phi$ is a function depending on $u$ and $|\nabla u|$, and $\gamma > 0$ is a small constant to be determined; one may as well follow the approach of Szekelyhidi~\cite{Szekelyhidi}.
Assume that $A$ is achieved at an interior point $z_0 \in M$ for some
$\xi \in T^{1,0}_{z_0} M$, $|\xi| = 1$.
We choose local coordinates around $z_0$  such that
$g_{i\bj} = \delta_{ij}$
and that $\fg_{i\bj}$ is diagonal
at $z_0$ with
\[ \fg_{1\bar 1} \geq \fg_{2 \bar 2} \geq \cdots \geq \fg_{n\bn}. \]
We shall assume $\fg_{1\bar{1}} \geq 1$;  otherwise we are done.

As pointed out in \cite{TWv13a}, an important fact is that when $\gamma$ is chosen sufficient small we have $\xi = \partial_1$
and $A =  e^{(1 + \gamma) \phi}  \fg_{1\bar 1}^{1 + \gamma} $ at $z_0$.
Indeed, this is obvious if  $\fg_{n\bn} \geq - \fg_{1\bar 1}$.
Suppose $\fg_{n\bn} <  - \fg_{1\bar 1}$. Then $n \geq 3$ and
 $(n-1) \fg_{1\bar 1} \geq - \fg_{n\bn}$ since
\[ \sum  \fg_{k\bk} \geq 0.\]
Clearly, $\xi_i = 0$ for $1 < i < n$
and $\xi_1^2 + \xi_n^2 = 1$. It follows that
\[ A e^{- (1 + \gamma) \phi} \leq \fg_{1\bar 1} \xi_1^2
     (\fg_{1\bar 1}^2 \xi_1^2 + \fg_{n\bn}^2 \xi_n^2)^{\frac{\gamma}{2}}
     \leq   \xi_1^2  (\xi_1^2 + (n-1)^2 \xi_n^2)^{\frac{\gamma}{2}}
          \fg_{1\bar 1}^{1 + \gamma} \leq \fg_{1\bar 1}^{1 + \gamma} \]
provided that $\gamma \leq \frac{2}{n (n-2)}$. This shows that $\xi_n = 0$.

Let $W = g_{1\bar{1}}^{-1} g^{k\bl} \fg_{1\bl} \fg_{k\bar{1}}$.
The function
$e^{(1+\gamma)  \phi} g_{1\bar{1}}^{-1} \fg_{1\bar{1}} W^{\frac{\gamma}{2}}$ which is  locally well defined
attains a maximum $A = (e^{\phi} \fg_{1\bar{1}})^{1+\gamma}$ at $z_0$.
It follows that
at $z_0$
\begin{equation}
\label{gblq-C80'}
 \begin{aligned}
\frac{\partial_i (g_{1\bar{1}}^{-1} \fg_{1\bar{1}})}{\fg_{1\bar{1}}}
   +  \frac{\gamma \partial_i W}{2 W} + (1+ \gamma)  \partial_i \phi = \,& 0, \\
\frac{\bpartial_i (g_{1\bar{1}}^{-1} \fg_{1\bar{1}})}{\fg_{1\bar{1}}}
   +\frac{ \gamma \bpartial_i W}{2 W} + (1+ \gamma) \bpartial_i \phi = \,& 0
\end{aligned}
\end{equation}
for each $1 \leq i \leq n$, and
\begin{equation}
\label{gblq-C90'}
\begin{aligned}
0 \geq \,&
   \frac{1}{\fg_{1\bar{1}}} F^{i\bi} \bpartial_i \partial_i (g_{1\bar{1}}^{-1} \fg_{1\bar{1}})
   - \frac{1}{\fg_{1\bar{1}}^2} F^{i\bi}
     \partial_i (g_{1\bar{1}}^{-1} \fg_{1\bar{1}}) \bpartial_i (g_{1\bar{1}}^{-1} \fg_{1\bar{1}}) \\
  & + \frac{\gamma}{2 W} F^{i\bi} \bpartial_i \partial_i W
     - \frac{\gamma}{2 W^2} F^{i\bi} \partial_i W \bpartial_i W
   + (1 + \gamma) F^{i\bi} \bpartial_i \partial_i \phi.
\end{aligned}
\end{equation}

Recall that in local coordinates
the Christoffel symbols $\Gamma_{ij}^k$ are defined by
\[ \nabla_{\frac{\partial}{\partial z_i}} \frac{\partial }{\partial z_j}
= \Gamma_{ij}^k \frac{\partial}{\partial z_k} \]
so
\[ \Gamma_{ij}^k = g^{k\bl} \frac{\partial g_{j\bl}}{\partial z_i}
= g^{k\bl} \partial_i g_{j\bl} \]
and
\[ \partial_i g_{j\bk} = g_{l\bk} \Gamma_{ij}^l, \;\;
\bpartial_i g_{j\bk} = g_{j\bl} \ol{\Gamma_{ik}^l}. \]
The torsion and curvature tensors are given by
\begin{equation}
\label{cma-K95}
    T^k_{ij} = \Gamma_{ij}^k  - \Gamma_{ji}^k
\end{equation}
and,  respectively,
\begin{equation}
\label{cma-K70}
\begin{aligned}
 R_{i\bj k\bl} = \,& - g_{m\bl}  \frac{\partial \Gamma_{ik}^m}{\partial \bz_j}
  = - \frac{\partial g_{k\bl}}{\partial z_i \partial \bz_j}
      +   g^{p\bq} \frac{\partial g_{k\bq}}{\partial z_i} \frac{\partial g_{p\bl}}{\partial \bz_j}.
\end{aligned} 
\end{equation}
By a lemma of Streets-Tian~\cite{ST11}, we may assume
$T_{ij}^k = 2 \Gamma_{ij}^k$ at $z_0$.

We record the following basic formulas
\[ \partial_i g^{k\bl} = - g^{k\bq} g^{p\bl} \partial_i g_{p\bq}
                                = - g^{p\bl} \Gamma_{ip}^k, \;
\bpartial_j g^{k\bl} = \ol{\partial_j g^{l\bk}}  = - g^{k\bq} \ol{\Gamma_{jq}^l}
\]
\[ \bpartial_j  \partial_i g^{k\bl} =  -  g^{p\bl} \bpartial_j \Gamma_{ip}^k
         + g^{p\bq} \Gamma_{ip}^k \ol{\Gamma_{jq}^l} \]
and
\[     \partial_i \fg_{k\bl}
    = \nabla_i \fg_{k\bl} + \Gamma_{ik}^m \fg_{m\bl},
    \;\;  \bpartial_j \fg_{k\bl}
    = \nabla_{\bj} \fg_{k\bl} + \ol{\Gamma_{jl}^m} \fg_{k\bm},   \]
\[    \bpartial_j \partial_i \fg_{k\bl}
    = \nabla_{\bj} \nabla_i \fg_{k\bl}
            + \ol{\Gamma_{jl}^q} \nabla_i \fg_{k\bq}
            + \Gamma_{ik}^m \nabla_{\bj} \fg_{m\bl}
             + \Gamma_{ik}^m \ol{\Gamma_{jl}^q} \fg_{m\bq}
            + \bpartial_j \Gamma_{ik}^m \fg_{m\bl}. \]
Therefore, for any indices $i, j, k, l, r, s$,
\begin{equation}
\label{gn-S85}
   \begin{aligned}
   \partial_i (g^{k\bl} \fg_{r\bs})
    = \,& g^{k\bl} \partial_i \fg_{r\bs} + \partial_i  g^{k\bl} \fg_{r\bs} \\
   = \,& g^{k\bl} \nabla_i \fg_{r\bs} +
       g^{k\bl} \Gamma_{ir}^m \fg_{m\bs}
       - g^{p\bl} \Gamma_{ip}^k \fg_{r\bs},
    \end{aligned}
  \end{equation}
    \[   \begin{aligned}
   \bpartial_j (g^{k\bl} \fg_{r\bs})
    = \,& g^{k\bl} \bpartial_j \fg_{r\bs} + \bpartial_j  g^{k\bl} \fg_{r\bs}
   = g^{k\bl} \nabla_{\bj} \fg_{r\bs} +
       g^{k\bl} \ol{\Gamma_{js}^m} \fg_{r\bm}
       - g^{k\bm} \ol{\Gamma_{jm}^l} \fg_{r\bs}
    \end{aligned} \]
  and
\begin{equation}
\label{gn-S90}
 \begin{aligned}
     \bpartial_j \partial_i (g^{k\bl} \fg_{r\bs})
    = \,& g^{k\bl} (\nabla_{\bj} \nabla_i \fg_{r\bs}
            + \ol{\Gamma_{js}^q} \nabla_i \fg_{r\bq}
            + \Gamma_{ir}^m \nabla_{\bj} \fg_{m\bs}
             + \Gamma_{ir}^m \ol{\Gamma_{js}^q} \fg_{m\bq}
            + \bpartial_j \Gamma_{ir}^m \fg_{m\bs}) \\
       & - g^{k\bq} \ol{\Gamma_{jq}^l}
                (\nabla_i \fg_{r\bs} + \Gamma_{ir}^m \fg_{m\bs})
              - g^{p\bl} \Gamma_{ip}^k (\nabla_{\bj} \fg_{r\bs}
             +  \ol{\Gamma_{js}^q} \fg_{r\bq}) \\
       &  -  (g^{p\bl} \bpartial_j \Gamma_{ip}^k -
          g^{p\bq} \Gamma_{ip}^k \ol{\Gamma_{jq}^l}) \fg_{r\bs}.
    \end{aligned}
  \end{equation}
Let $r = k$ 
and sum over $k$,
\begin{equation}
\label{gn-S95}
  \partial_i (g^{k\bl} \fg_{k\bs}) = g^{k\bl} \nabla_i \fg_{k\bs},
 \end{equation}
 \begin{equation}
\label{gn-S96}
   \bpartial_j (g^{k\bl} \fg_{k\bs}) = g^{k\bl} \nabla_{\bj} \fg_{k\bs} +
       g^{k\bl} \ol{\Gamma_{js}^m} \fg_{k\bm} - g^{k\bm} \ol{\Gamma_{jm}^l} \fg_{k\bs},
 \end{equation}
 \begin{equation}
\label{gn-S105}
\begin{aligned}
\bpartial_j \partial_i (g^{k\bl} \fg_{k\bs})
     = \,& g^{k\bl} \nabla_{\bj} \nabla_i \fg_{k\bs}
          + g^{k\bl} \ol{\Gamma_{js}^t} \nabla_i \fg_{k\bt}
           - g^{k\bt} \ol{\Gamma_{jt}^l} \nabla_i \fg_{k\bs}
           \end{aligned}
  \end{equation}
which also follows from 
\[ \begin{aligned}
\bpartial_j \partial_i (g^{k\bl} \fg_{k\bs})
     = \,& \bpartial_j (g^{k\bl} \nabla_i \fg_{k\bs}) \\
     = \,& \bpartial_j g^{k\bl} \nabla_i \fg_{k\bs}
          +  g^{k\bl} \bpartial_j \nabla_i \fg_{k\bs} \\
     = \,& g^{k\bl} \nabla_{\bj} \nabla_i \fg_{k\bs}
          + g^{k\bl} \ol{\Gamma_{js}^t} \nabla_i \fg_{k\bt}
           - g^{k\bt} \ol{\Gamma_{jt}^l} \nabla_i \fg_{k\bs}.
           \end{aligned}  \]

We shall also need
  \[  \begin{aligned}
     \partial_i (g_{1\bar{1}}^{-1} \fg_{1\bl})
     = \,& g_{1\bar{1}}^{-1}  \partial_i  \fg_{1\bl}
      -  g_{1\bar{1}}^{-2}  \fg_{1\bl} \partial_i  g_{1\bar{1}} \\
      = \,& g_{1\bar{1}}^{-1} \nabla_i \fg_{1\bl}
       + g_{1\bar{1}}^{-2} \Gamma_{i1}^m (g_{1\bar{1}} \fg_{m\bl}
      -  g_{m\bar{1}}  \fg_{1\bl}),
     \end{aligned}    \]
 \[  \begin{aligned}
     \partial_{\bj} (g_{1\bar{1}}^{-1} \fg_{1\bl})
     = \,& g_{1\bar{1}}^{-1}  \partial_{\bj}  \fg_{1\bl}
      -  g_{1\bar{1}}^{-2}  \fg_{1\bl} \partial_{\bj}  g_{1\bar{1}} \\
      = \,& g_{1\bar{1}}^{-1} (\nabla_{\bj} \fg_{1\bl} + \ol{\Gamma_{jl}^m} \fg_{1\bm})
       - g_{1\bar{1}}^{-2} \ol{\Gamma_{j1}^m} g_{1\bm}  \fg_{1\bl}
     \end{aligned}    \]
and
      \[  \begin{aligned}
\bpartial_j  \partial_i (g_{1\bar{1}}^{-1} \fg_{1\bar{1}})
      = \,& g_{1\bar{1}}^{-1} \nabla_{\bj} \nabla_i \fg_{1\bar{1}}
      + g_{1\bar{1}}^{-2} (g_{1\bar{1}} \ol{\Gamma_{j1}^m} \nabla_i \fg_{1\bm} - g_{1\bm} \ol{\Gamma_{j1}^m} \nabla_i \fg_{1\bar{1}}) \\
      & + g_{1\bar{1}}^{-2} \Gamma_{i1}^m
         (g_{1\bar{1}} \nabla_{\bj} \fg_{m\bar{1}} - g_{m\bar{1}}  \nabla_{\bj} \fg_{1\bar{1}}) \\
  &  + \bpartial_j (g_{1\bar{1}}^{-2} \Gamma_{i1}^m) (g_{1\bar{1}} \fg_{m\bar{1}}
      -  g_{m\bar{1}}  \fg_{1\bar{1}})
      + g_{1\bar{1}}^{-2} \Gamma_{i1}^m (\bpartial_j g_{1\bar{1}} \fg_{m\bar{1}}
      -  \bpartial_j g_{m\bar{1}}  \fg_{1\bar{1}}).
     \end{aligned}    \]

At $z_0$ where $g_{i\bj} = \delta_{ij}$, $T_{ij}^k = 2 \Gamma_{ij}^k$, and
$\fg_{i\bj}$ is diagonal, we have 
\[   \partial_i (g^{k\bl} \fg_{k\bs}) = \nabla_i \fg_{l\bs}, \]
\[ \bpartial_j (g^{k\bl} \fg_{k\bs}) = \nabla_{\bj} \fg_{l\bs} +
         \ol{\Gamma_{js}^l} (\fg_{l\bl} - \fg_{s\bs}), \]
\[  \bpartial_j \partial_i (g^{k\bl} \fg_{k\bs})
     = \nabla_{\bj} \nabla_i \fg_{l\bs}
       +  \ol{\Gamma_{js}^t} \nabla_i \fg_{l\bt}
        - \ol{\Gamma_{jt}^l} \nabla_i \fg_{t\bs}, \]
\[  \begin{aligned}
     \partial_i (g_{1\bar{1}}^{-1} \fg_{1\bl})
      = \,& \nabla_i \fg_{1\bl}
       + (\Gamma_{i1}^l \fg_{l\bl}
      -  \Gamma_{i1}^1 \fg_{1\bl}),
     \end{aligned}   \]
 \[  \begin{aligned}
     \partial_{\bj} (g_{1\bar{1}}^{-1} \fg_{1\bl})
     = \,& \nabla_{\bj} \fg_{1\bl} + \ol{\Gamma_{jl}^1} \fg_{1\bar{1}}
       - \ol{\Gamma_{j1}^1}  \fg_{1\bl}.
     \end{aligned}    \]
In particular,
\[ \partial_i (g_{1\bar{1}}^{-1} \fg_{1\bar{1}}) = \nabla_i \fg_{1\bar{1}}, \;\;
\bpartial_j (g_{1\bar{1}}^{-1} \fg_{1\bar{1}}) = \nabla_{\bj} \fg_{1\bar{1}}, \]
and
  \[  \begin{aligned}
\bpartial_j  \partial_i (g_{1\bar{1}}^{-1} \fg_{1\bar{1}})
       = \,& \nabla_{\bj} \nabla_i \fg_{1\bar{1}}
    + (\ol{\Gamma_{j1}^m} \nabla_i \fg_{1\bm}
        - \ol{\Gamma_{j1}^1} \nabla_i \fg_{1\bar{1}}) \\
      & + (\Gamma_{i1}^m \nabla_{\bj} \fg_{m\bar{1}}
            - \Gamma_{i1}^1 \nabla_{\bj} \fg_{1\bar{1}})
      + (\Gamma_{i1}^1  \ol{\Gamma_{j1}^1} - \Gamma_{i1}^m \ol{\Gamma_{j1}^m}) \fg_{1\bar{1}}.
     \end{aligned}    \]
It follows that
\begin{equation}
\label{gn-S115}
\begin{aligned}
 \partial_i W 
 = \,& g_{1\bar{1}}^{-1}  \fg_{1\bl} \partial_i (g^{k\bl} \fg_{k\bar{1}})
      + g^{k\bl} \fg_{k\bar{1}} \partial_i (g_{1\bar{1}}^{-1} \fg_{1\bl})
 = 2 \fg_{1\bar{1}} \nabla_i \fg_{1\bar{1}},
     \end{aligned}
      \end{equation}
and
\[ \begin{aligned}
\bpartial_j \partial_i W
    = \,& g_{1\bar{1}}^{-1}  \fg_{1\bl} \bpartial_j \partial_i  (g^{k\bl} \fg_{k\bar{1}})
            + \bpartial_j (g^{k\bl} \fg_{k\bar{1}}) \partial_i  (g_{1\bar{1}}^{-1}  \fg_{1\bl}) \\
         & +  \bpartial_j (g_{1\bar{1}}^{-1} \fg_{1\bl})  \partial_i (g^{k\bl} \fg_{k\bar{1}})
             + g^{k\bl} \fg_{k\bar{1}}  \bpartial_j \partial_i  (g_{1\bar{1}}^{-1} \fg_{1\bl}) \\
        = \,& \fg_{1\bar{1}} (\nabla_{\bj} \nabla_i \fg_{1\bar{1}}
       +  \ol{\Gamma_{j1}^t} \nabla_i \fg_{1\bt}
        - \ol{\Gamma_{jt}^1} \nabla_i \fg_{t\bar{1}}) \\
        &     +  (\nabla_{\bj} \fg_{l\bar{1}} +
        \ol{\Gamma_{j1}^l} \fg_{l\bl} - \ol{\Gamma_{j1}^l} \fg_{1\bar{1}})(\nabla_i \fg_{1\bl}
       + \Gamma_{i1}^l \fg_{l\bl}
      -  \Gamma_{i1}^1 \fg_{1\bl}) \\
         &  +   \nabla_i \fg_{l\bar{1}}
         (\nabla_{\bj} \fg_{1\bl} + \ol{\Gamma_{jl}^1} \fg_{1\bar{1}}
         - \ol{\Gamma_{j1}^1} \fg_{1\bl}) \\
    & + \fg_{1\bar{1}} [\nabla_{\bj} \nabla_i \fg_{1\bar{1}}
    + (\ol{\Gamma_{j1}^m} \nabla_i \fg_{1\bm}
        - \ol{\Gamma_{j1}^1} \nabla_i \fg_{1\bar{1}}) \\
      & + (\Gamma_{i1}^m \nabla_{\bj} \fg_{m\bar{1}}
            - \Gamma_{i1}^1 \nabla_{\bj} \fg_{1\bar{1}})
      + (\Gamma_{i1}^1  \ol{\Gamma_{j1}^1} - \Gamma_{i1}^m \ol{\Gamma_{j1}^m}) \fg_{1\bar{1}}].
                 \end{aligned} \]
After some cancellations, this can be rewritten as
\begin{equation}
\label{gn-S110}
 \begin{aligned}
\bpartial_j \partial_i W
   = \,& 2 \fg_{1\bar{1}} \nabla_{\bj} \nabla_i \fg_{1\bar{1}}
   + 2 \nabla_i \fg_{1\bar{1}} \nabla_{\bj} \fg_{1\bar{1}}
   + \sum_{l>1} \nabla_i \fg_{l\bar{1}} \nabla_{\bj} \fg_{1\bl}    \\
     &  + \sum_{l > 1} (\nabla_i \fg_{1\bl} + {\Gamma_{i1}^l} \fg_{l\bl})
       (\nabla_{\bj} \fg_{l\bar{1}} + \ol{\Gamma_{j1}^l} \fg_{l\bl}) \\
   & + \fg_{1\bar{1}} \sum_{l > 1} (\ol{\Gamma_{j1}^l} \nabla_i \fg_{1\bl}
       + \Gamma_{i1}^l  \nabla_{\bj} \fg_{l\bar{1}}) \\
  &    -  \fg_{1\bar{1}} \sum_{l > 1} \Gamma_{i1}^m \ol{\Gamma_{j1}^m} 
        (\fg_{1\bar{1}} +  \fg_{l\bl)}.
              \end{aligned}
 \end{equation}

Finally, we obtain
\begin{equation}
\label{gblq-C103}
\begin{aligned}
F^{i\bi} \partial_i W \bpartial_i W
= 4 \fg_{1\bar{1}}^2 F^{i\bi} \nabla_i \fg_{1\bar{1}} \nabla_{\bi} \fg_{1\bar{1}},
\end{aligned}
\end{equation}
\begin{equation}
\label{gblq-C105}
\begin{aligned}
F^{i\bi} \partial_i (g_{1\bar{1}}^{-1} \fg_{1\bar{1}}) \bpartial_i (g_{1\bar{1}}^{-1} \fg_{1\bar{1}})
= F^{i\bi} \nabla_i \fg_{1\bar{1}} \nabla_{\bi} \fg_{1\bar{1}}
\end{aligned}
\end{equation}
and, by Cauchy-Schwarz inequality,
\begin{equation}
\label{gblq-C100}
 \begin{aligned}
F^{i\bi} \bpartial_i \partial_i W
\geq \,& 2 \fg_{1\bar{1}} F^{i\bi} \nabla_{\bi} \nabla_i \fg_{1\bar{1}}
    + 2 F^{i\bi}  \nabla_i \fg_{1\bar{1}} \nabla_{\bi} \fg_{1\bar{1}} \\
  + \,& \sum_{l > 1} F^{i\bi} \nabla_i \fg_{1\bl} \nabla_{\bi} \fg_{l\bar{1}}
  + \frac{1}{2} \sum_{l > 1} F^{i\bi} \nabla_i \fg_{1\bl} \nabla_{\bi} \fg_{l\bar{1}}
   - C \fg_{1\bar{1}}^2 \sum F^{i\bi},
 \end{aligned}
 \end{equation}
\begin{equation}
\label{gblq-C101}
 \begin{aligned}
F^{i\bi} \bpartial_i  \partial_i (g_{1\bar{1}}^{-1} \fg_{1\bar{1}})
\geq \,& F^{i\bi} \nabla_{\bi} \nabla_i \fg_{1\bar{1}}
  - \frac{\gamma}{8 \fg_{l\bar{1}}} \sum_{l > 1} F^{i\bi}
            \nabla_i \fg_{1\bl} \nabla_{\bi} \fg_{l\bar{1}}
        - C \fg_{1\bar{1}} \sum F^{i\bi}.
 \end{aligned}
 \end{equation}

 In summary, we can rewrite \eqref{gblq-C80'} as
 \begin{equation}
\label{gblq-C80''}
 \begin{aligned}
\nabla_i \fg_{1\bar{1}} + \fg_{1\bar{1}} \partial_i \phi = 0, \;
\nabla_{\bi} \fg_{1\bar{1}} + \fg_{1\bar{1}} \bpartial_i \phi = 0,
\end{aligned}
\end{equation}
and, plugging \eqref{gblq-C103}-\eqref{gblq-C101} into \eqref{gblq-C90'},
we derive
 \begin{equation}
\label{gblq-C91}
\begin{aligned}
0  \geq \,& \frac{1}{\fg_{1\bar{1}}}  F^{i\bi} \nabla_{\bi} \nabla_i \fg_{1\bar{1}}
 - \frac{1}{\fg_{1\bar{1}}^2} F^{i\bi}  \nabla_i \fg_{1\bar{1}} \nabla_{\bi} \fg_{1\bar{1}}
   + F^{i\bi} \bpartial_i \partial_i \phi \\
 & + \frac{\gamma}{\fg_{1\bar{1}}^2} \sum_{l > 1} F^{i\bi} \nabla_i \fg_{1\bl} \nabla_{\bi} \fg_{l\bar{1}} + \frac{\gamma}{16 \fg_{1\bar{1}}^2} \sum_{l > 1} F^{i\bi}
            \nabla_i \fg_{1\bl} \nabla_{\bi} \fg_{l\bar{1}}
        - C \sum F^{i\bi}.
\end{aligned}
\end{equation}

Next, differentiate equation~\eqref{cma2-M10} twice to obtain (at $z_0$),
\begin{equation}
\label{gblq-C74}
 F^{i\bi}  \nabla_k \fg_{i\bi} = \nabla_k \psi,
 \end{equation}
\begin{equation}
\label{gblq-C75}
  F^{i\bi}   \nabla_{\bar{1}} \nabla_1 \fg_{i\bi} + F^{i\bj, k\bl} \nabla_1 \fg_{i\bj} \nabla_{\bar{1}} \fg_{k\bl}
   = \nabla_{\bar{1}} \nabla_1 \psi.
\end{equation}

Recall the formulas for communication of covariant derivatives
\begin{equation}
\label{gblq-B147}
\left\{ \begin{aligned}
 u_{i \bj k} - u_{k \bj i} = \,& T_{ik}^l u_{l\bj}, \;\;
u_{i \bj k} - u_{i k \bj} = - g^{l\bm} R_{k \bj i \bm} u_l, \\
u_{i\bj k\bl} - u_{i\bj \bl k}
      = \,& g^{p\bq} R_{k\bl i\bq} u_{p\bj}
          - g^{p\bq} R_{p \bl k \bj} u_{i\bq}, \\
u_{i \bj k \bl} - u_{k \bl i \bj}
  = \,&  g^{p\bq} (R_{k\bl i\bq} u_{p\bj} - R_{i\bj k\bq} u_{p\bl})
        + T_{ik}^p u_{p\bj \bl} + \ol{T_{jl}^q} u_{i\bq k}
        - T_{ik}^p \ol{T_{jl}^q} u_{p\bq}
 \end{aligned}  \right.
\end{equation}
where for simplicity, $u_{i\bj} = \nabla_{\bj} \nabla_i u = \partial_{\bj} \partial_i u$,
\[  \begin{aligned}
u_{i\bj k} =  \nabla_k  u_{i\bj} = \partial_{k}  u_{i\bj} - \Gamma_{ki}^l u_{l\bj},
   \end{aligned} \]
and
   \[  \begin{aligned}
u_{i\bj k\bl} =  \nabla_{\bl}  u_{i\bj k} = \partial_{\bl}  u_{i\bj k} -
\ol{\Gamma_{lj}^m} u_{i\bm k}.
   \end{aligned} \]
Therefore at $z_0$, by \eqref{gblq-B147}, 
\begin{equation}
\label{gblq-R155}
 \begin{aligned}
 \nabla_{\bi} \nabla_i \fg_{1\bar{1}}  -  \nabla_{\bar{1}} \nabla_1 \fg_{i\bi}
   = \,&  R_{i\bi 1\bar{1}} \fg_{1\bar{1}} - R_{1\bar{1} i\bi} \fg_{i\bi}
         - T_{i1}^l \nabla_{\bi} \fg_{l\bar{1}}  - \ol{T_{i1}^l} \nabla_i \fg_{1\bl} \\
     &  - T_{i1}^l  \ol{T_{i1}^l} \fg_{l\bl} + H_{i\bi}
  \end{aligned}
 \end{equation}
where
\[ \begin{aligned}
    H_{i\bi} = \,& \nabla_{\bi} \nabla_i \chi_{1\bar{1}}
                       -  \nabla_{\bar{1}} \nabla_1 \chi_{i\bi}
      - 2 \fRe\{T_{i1}^l \nabla_{\bi} \chi_{l\bar{1}}\} 
      + R_{i\bi 1\bl} \chi_{l\bar{1}} - R_{1\bar{1} i\bl} \chi_{l\bi}
       - T_{i1}^j  \ol{T_{i1}^l} \chi_{j\bl}.
  \end{aligned} \]
 It follows from Schwarz inequality that 
\begin{equation}
\label{gblq-R156}
 \begin{aligned}
F^{i\bi} \nabla_{\bi} \nabla_i \fg_{1\bar{1}}
   \geq \,& F^{i\bi} \nabla_{\bar1} \nabla_1\fg_{i\bi}
         -  \frac{\gamma}{16 \fg_{1\bar{1}}} 
             F^{i\bi}  \nabla_i \fg_{1\bl} \nabla_{\bi} \fg_{l\bar{1}} \\
       &  - C \fg_{1\bar{1}} \sum F^{i\bi} + F^{i\bi} H_{i\bi}.
 \end{aligned}
 \end{equation}
Combining \eqref{gblq-C91}, \eqref{gblq-C75}  and \eqref{gblq-R156},
by Schwarz inequality we derive
 \begin{equation}
\label{gblq-C90'''}
\begin{aligned}
\fg_{1\bar{1}} F^{i\bi} \bpartial_i \partial_i \phi
   \leq  \,&  - \nabla_{\bar{1}} \nabla_1 \psi - E - F^{i\bi} H_{i\bi}
                  +  C \fg_{1\bar{1}} \sum F^{i\bi} \\
      & -  \frac{\gamma}{32 \fg_{1\bar{1}}} \sum_{l > 1} F^{i\bi}
            (\nabla_i \fg_{l\bar{1}} \nabla_{\bi} \fg_{1\bl} + \nabla_i \fg_{1\bl} \nabla_{\bi} \fg_{l\bar{1}})
 \end{aligned}
\end{equation}
where
\[ E = - F^{i\bj, k\bl} \nabla_1 \fg_{i\bj} \nabla_{\bar{1}} \fg_{k\bl}
         -   \frac{1+\gamma}{\fg_{1\bar{1}}} F^{i\bi}  \nabla_i \fg_{1\bar{1}} \nabla_{\bi} \fg_{1\bar{1}}. \]

In the rest of this section we shall not need the last nonpositive term on the right hand side of \eqref{gblq-C90'''} and therefore drop it.

\subsection{The term $E$}

To estimate the term $E$, set
\[  \begin{aligned}
J = \,& \{i: |\fg_{i\bi}| \geq \gamma \fg_{1\bar{1}}\}, \\
K = \,& \{i:  |\fg_{i\bi}| < \gamma \fg_{1\bar{1}}, \; \gamma F^{i\bi} > F^{1\bar{1}}\}, \\
L  = \,& \{i:  |\fg_{i\bi}| < \gamma \fg_{1\bar{1}}, \; \gamma F^{i\bi} \leq F^{1\bar{1}}\}
  \end{aligned} \]
  where $\gamma > 0$ is same as in \eqref{gblq-C10A}.
By an inequality due to Caffarelli-Nirenberg-Spruck, Andrews 
and Gerhardt 
(see e.g. \cite{Spruck05}) we have
\[ - F^{i\bj, k\bl} \nabla_1 \fg_{i\bj} \nabla_{\bar{1}} \fg_{k\bl}
 \geq  \sum_{i \neq j}
     \frac{F^{i\bi} - F^{j\bj}}{\fg_{j\bj} - \fg_{i\bi}} |\nabla_1 \fg_{i\bj}|^2
 \geq \sum_{i \geq 2}  \frac{F^{i\bi} - F^{1\bar{1}}}{\fg_{1\bar{1}} - \fg_{i\bi}}
    |\nabla_1 \fg_{i\bar{1}}|^2. \]
By the first formula in \eqref{gblq-B147},
\[ \nabla_1 \fg_{i\bar{1}}
    = \nabla_i \fg_{1\bar{1}} + T_{i1}^1 \fg_{1\bar{1}} + \kappa_i \]
where
\[ \begin{aligned}
 \kappa_i
  = \,& \nabla_1 \chi_{i\bar{1}} - \nabla_i  \chi_{1\bar{1}} - T_{i1}^l \chi_{l\bar{1}} \\
  = \,& \chi_{i\bar{1} 1} + \chi_{i\bar{1}, \zeta_{\alpha}} \p_1 \p_{\alpha} u - \chi_{1\bar{1} i} - \chi_{1\bar{1}, \zeta_{\alpha}} \p_i \p_{\alpha} u - T_{i1}^l \chi_{l\bar{1}}.
 \end{aligned}  \]
  This yields
\begin{equation}
\label{gsz-G270}
\begin{aligned}
- F^{i\bj, k\bl} \nabla_1 \fg_{i\bj} \nabla_{\bar{1}} \fg_{k\bl}
 \geq \,& \frac{1-\gamma}{(1+\gamma) \fg_{1\bar{1}}}
            \sum_{i \in K} F^{i\bi} |\nabla_i \fg_{1\bar{1}} + T_{i1}^1 \fg_{1\bar{1}}
           + \kappa_i|^2 \\
 \geq \,& \frac{1- 2 \gamma}{(1+\gamma) \fg_{1\bar{1}}}
            \sum_{i \in K} F^{i\bi} |\nabla_i \fg_{1\bar{1}}|^2
            - C \fg_{1\bar{1}} \sum F^{i\bi} \\
   & - \frac{C}{\fg_{1\bar{1}}} \sum F^{i\bi} (|\chi_{i \bar{1}, \zeta_{\alpha}}
   \p_1 \p_{\alpha} u|^2 + |\p_i \p_{\alpha} u|^2).
\end{aligned}
\end{equation}
Therefore,
\begin{equation}
\label{gblq-R350}
 \begin{aligned}
  E \geq \,& - 4 \gamma \fg_{1\bar{1}}
            \sum_{i \in K} F^{i\bi} |\nabla_i \phi|^2
            -  (1 + \gamma) \fg_{1\bar{1}}  \sum_{J \cup L} F^{i\bi}  |\nabla_i \phi|^2 \\
       & - \frac{C}{\fg_{1\bar{1}}} \sum_{i, \alpha} F^{i\bi} (|\chi_{i \bar{1}, \zeta_{\alpha}}
       \nabla_1 \nabla_{\alpha} u|^2 + |\nabla_i \nabla_{\alpha} u|^2)
              - C \fg_{1\bar{1}} \sum F^{i\bi}.
    \end{aligned}
\end{equation}

\subsection{ The term $H = F^{i\bi} H_{i\bi}$}
 We need to handle the first three terms in $H_{i\bi}$ carefully.
 First, by \eqref{gn-R60},
\begin{equation}
\label{gn-C2100}
\begin{aligned}
    F^{i\bi}  \fRe\{T_{i1}^l \nabla_{\bi} \chi_{l\bar{1}}\}
          \leq  C \sum F^{i\bi} |\nabla_i \nabla_{\alpha} u| + C \fg_{1\bar{1}} \sum F^{i\bi},
    \end{aligned}
\end{equation}
Next, a straightforward calculation using \eqref{gn-R60}-\eqref{gn-R80} shows
\begin{equation}
\label{gn-C2110}
  \begin{aligned}
 \nabla_{\bl} \nabla_k {\chi}_{i\bj}
    = \,& \nabla_{\bl} (\chi_{i\bj k} + \chi_{i\bj, \zeta_{\alpha}} \p_k \p_{\alpha } u
    + \chi_{i\bj, \bzeta_{\alpha}} \p_k \p_{\balpha} u) \\
 = \,& \chi_{i\bj k\bl} + \chi_{i\bj k, \zeta_{\alpha}} \p_{\bl} \p_{\alpha} u
          + \chi_{i\bj k, \bzeta_{\alpha}} \p_{\bl} \p_{\balpha } u \\
      & + (\chi_{i\bj \bl, \zeta_{\alpha}}
              + \chi_{i\bj, \zeta_{\alpha} \zeta_{\beta}} \p_{\bl} \p_{\beta} u
              + \chi_{i\bj, \zeta_{\alpha} \bzeta_{\beta}} \p_{\bl} \p_{\bbeta} u)
                 \p_k \p_{\alpha } u \\
      & + (\chi_{i\bj \bl, \bzeta_{\alpha}}
              + \chi_{i\bj, \bzeta_{\alpha} \zeta_{\beta}} \p_{\bl} \partial_{\beta} u
              + \chi_{i\bj, \bzeta_{\alpha} \bzeta_{\beta}} \p_{\bl} \p_{\bbeta} u)
                 \p_k \p_{\balpha} u  \\
     &  + \chi_{i\bj, \zeta_{\alpha}} \nabla_{\bl}
             (\nabla_k \nabla_{\alpha} u + \Gamma_{k \alpha}^m \nabla_m u)
         + \chi_{i\bj, \bzeta_{\alpha}} \nabla_{\bl} \nabla_k \nabla_{\balpha} u \\
= \,& \chi_{i\bj k \bl} + \chi_{i\bj k, \zeta_{\alpha}} \p_{\bl} \p_{\alpha} u
       + \chi_{i\bj k, \bzeta_{\alpha}} \p_{\bl} \p_{\balpha } u
       + \chi_{i\bj \bl, \zeta_{\alpha} } \p_k \p_{\alpha } u
       + \chi_{i\bj \bl, \bzeta_{\alpha} } \p_k \p_{\balpha} u \\
      & + \chi_{i\bj, \zeta_{\alpha} \zeta_{\beta}} \p_ku_{\alpha  } \p_{\bl} \p_{\beta} u
         + \chi_{i\bj, \zeta_{\alpha} \bzeta_{\beta}}
              \p_k \p_ {\alpha} u \p_{\bl} \p_{\bbeta}  u
         + \chi_{i\bj, \bzeta_{\alpha} \zeta_{\beta}}
              \p_k \p_ { \balpha} u \p_{\bl} \partial_{\beta } u \\
    & + \chi_{i\bj, \bzeta_{\alpha} \bzeta_{\beta}}
              \p_k \p_ { \balpha} u \p_{\bl} \p_{\bbeta} u
       + \chi_{i\bj, \zeta_{\alpha}} \nabla_{\bl} \nabla_k \nabla_{\alpha} u
       + \chi_{i\bj, \bzeta_{\alpha}} \nabla_{\bl} \nabla_k \nabla_{\balpha} u \\
    & + \chi_{i\bj, \zeta_{\alpha}} (\nabla_{\bl} \Gamma_{k \alpha}^m \nabla_m u
          + \Gamma_{k \alpha}^m \nabla_{\bl} \nabla_m u).
  \end{aligned}
  \end{equation}
By \eqref{gn-R60},
\[ \nabla_k \nabla_{\bi} \nabla_i u = \nabla_k \fg_{i \bi} - \nabla_k \chi_{i \bi}
= \nabla_k \fg_{i \bi} - \chi_{i \bi k} - \chi_{i \bi, \zeta_{\alpha}} \p_k \p_{\alpha} u
     - \chi_{i \bi, \bar{\zeta}_{\alpha}} \p_k \p_{\balpha} u, \]
\[ \nabla_{\bk} \nabla_{\bi} \nabla_i u
   = \nabla_{\bk} \fg_{i \bi} - \nabla_{\bk} \chi_{i \bi}
   = \nabla_{\bk} \fg_{i \bi} - \chi_{i \bi {\bk}} - \chi_{i \bi, \zeta_{\alpha}} \p_{\bk} \p_{\alpha} u
     - \chi_{i \bi, \bar{\zeta}_{\alpha}} \p_{\bk} \p_{\balpha} u. \]
Therefore  by 
 \eqref{gblq-C80''},  \eqref{gblq-C74} and  \eqref{gblq-B147}, 
\begin{equation}
\label{gn-C2120}
\begin{aligned}
F^{i\bi} \nabla_{\bi} \nabla_i \chi_{1\bar{1}}
\geq \,& 2 F^{i\bi} \fRe\{\chi_{1 \bar{1}, \zeta_{\alpha}}
                  \nabla_{\alpha} \nabla_{\bi} \nabla_i   u\}
                  - C \sum_{i, \alpha} F^{i\bi} |\p_i  \p_{\alpha} u|^2
                  - C F^{i\bi}  (\fg_{i\bi}^2 + 1) \\
 \geq \,&  2 \fRe\{\chi_{1 \bar{1}, \zeta_{\alpha}}  \nabla_{\alpha} \psi\}
                - 2 F^{i\bi} \fRe\{\chi_{1 \bar{1}, \zeta_{\alpha}} \chi_{i\bi, \zeta_{\beta}}
                 \p_{\alpha} \p_{\beta} u\} \\
        & - C \sum_{i, \alpha} F^{i\bi} |\nabla_i  \nabla_{\alpha} u|^2
             - C F^{i\bi}  (\fg_{i\bi}^2 + 1), 
\end{aligned}
\end{equation}
\begin{equation}
\label{gblq-R155.4'}
\begin{aligned}
F^{i\bj} \nabla_{\bar{1}} \nabla_1 \chi_{i\bi}
 \leq \,& 2 F^{i\bi} \fRe\{\chi_{i \bi, \zeta_{\alpha}} \nabla_{\alpha} \nabla_{\bar{1}} \nabla_1 u + \chi_{i \bi \bar{1}, \zeta_{\alpha}} \p_1  \p_{\alpha} u\}
  + C  \fg_{1\bar{1}} \sum F^{i\bi}  + Q \\
 = \,& - 2 \fg_{1\bar{1}} F^{i\bi} \fRe\{\chi_{i \bi, \zeta_{\alpha}}  \nabla_{\alpha} \phi\}
          -  2 F^{i\bi} \fRe\{\chi_{i\bi, \zeta_{\alpha}} \chi_{1 \bar{1}, \zeta_{\beta}}
               \p_{\beta}  \p_{\alpha}  u\} \\
   &        +  2 F^{i\bi} \fRe\{\chi_{i \bi \bar{1}, \zeta_{\alpha}} \p_1  \p_{\alpha} u\}
      + C  \fg_{1\bar{1}} \sum F^{i\bi}  + Q
     \end{aligned}
\end{equation}
where
\begin{equation}
\label{gblq-R155.6}
 \begin{aligned}
   Q  
         = \,& F^{i\bi} \chi_{i \bi, \zeta_{\alpha} \bar{\zeta}_{\beta}}
 (\p_1  \p_{\alpha} u  \p_{\bar{1}} \p_{\bbeta} u + \p_{\bar{1}} \p_{\alpha} u  \p_{{1}} \p_{\bbeta} u)      \\
  & + 2 F^{i\bi} \fRe\{\chi_{i\bi, \zeta_{\alpha} \zeta_{\beta}}
                 \p_1  \p_{\alpha} u  \p_{\bar{1}} \p_{\beta} u\}.
   \end{aligned}
\end{equation}
It follows that 
\begin{equation}
\label{gblq-R155.4}
\begin{aligned}
F^{i\bi} H_{i \bi} \geq \,&  2 \fRe\{\chi_{1 \bar{1}, \zeta_{\alpha}}  \nabla_{\alpha} \psi\}
   +  2 \fg_{1\bar{1}} F^{i\bi} \fRe\{\chi_{i \bi, \zeta_{\alpha}}  \nabla_{\alpha} \phi\} \\
   & - 2 F^{i\bi} \fRe\{\chi_{i\bi  \bar{1}, \zeta_{\alpha}} \p_1  \p_{\alpha} u\}
      - C F^{i\bi}  \fg_{i\bi}^2 - Q \\
    & - C \sum_{i, k} F^{i\bi} |\nabla_i  \nabla_{k} u|^2
        - C \fg_{1\bar{1}} \sum F^{i\bi}.
   \end{aligned}
\end{equation}

\subsection{The function $\phi$ }

We choose $\phi = \eta - \log (1 - b |\nabla u|^2)$ where $\eta$ is a function to be determined,
and $b$ is a positive constant satisfying
$2 b |\nabla u|^2  \leq 1$.
Write
\[ h = \log (1 - b |\nabla u|^2). \]
By straightforward calculations,
\begin{equation}
\label{hess-a264}
  \begin{aligned}
- b^{-1} e^{h} \p_i h = \,& \nabla_k u \nabla_i \nabla_{\bk} u
    + \nabla_{\bk} u \nabla_i \nabla_k u, \\
    - b^{-1} e^{h} \p_{\bi} h = \,& \nabla_k u \nabla_{\bi} \nabla_{\bk} u
    + \nabla_{\bk} u \nabla_{\bi} \nabla_k u
  \end{aligned}
\end{equation}
and
\begin{equation}
\label{hess-a265}
  \begin{aligned}
- b^{-1} e^{h} (\p_{\bi} h \p_i h  + \p_{\bi} \p_i h)
  = \,& \nabla_{\bk} u \nabla_{\bi} \nabla_i \nabla_k u
          + \nabla_{k} u \nabla_{\bi} \nabla_i \nabla_{\bk} u \\
      & + \sum (\nabla_i \nabla_{\bk} u \nabla_k \nabla_{\bi} u +
         \nabla_i \nabla_{k} u \nabla_{\bi} \nabla_{\bk} u)   \\
    = \,& \nabla_{\bk} u \nabla_k \nabla_{\bi} \nabla_i u +
             \nabla_k u \nabla_{\bk} \nabla_{\bi} \nabla_i u
          + 2 \fRe\{T_{ik}^l \nabla_{\bi} \nabla_l u)\} \\
     & + \sum_k (|\nabla_i \nabla_{\bk} u|^2 + |\nabla_i \nabla_{k} u|^2).
\end{aligned}
\end{equation}
Therefore, by \eqref{gn-R60},
\begin{equation}
\label{hess-a271.1}
 \begin{aligned}
 - F^{i\bi} \p_{\bi} \p_i h \geq
     \,& F^{i\bi} \p_i h \p_{\bi} h + 2 F^{i\bi} \fRe\{\chi_{i \bi, \zeta_{\alpha}} \p_{\alpha} h\}
          + 2 b e^{-h} \fRe\{\nabla_{\bk} u \nabla_k \psi \}\\
       &  +  \frac{b}{2} F^{i\bi} \fg_{i\bi}^2
           + b \sum_{i,k} F^{i\bi}  |\nabla_i \nabla_{k} u|^2 - C b\sum F^{i\bi}.
\end{aligned}
\end{equation}

Let $\cL$ be the linear operator given by
\[ \cL v = F^{i\bi} \p_{\bi} \p_i v + 2 F^{i\bi} \fRe\{\chi_{i \bi, \zeta_{\alpha}} \p_{\alpha} v \}. \]
By \eqref{hess-a271.1} we derive
\begin{equation}
\label{hess-a271.6}
\begin{aligned}
\cL \phi \geq \,&  \cL \eta + \frac{b}{2} F^{i\bi} \fg_{i\bi}^2
             + b \sum_{i,k} F^{i\bi} |\nabla_i \nabla_{k} u|^2
               - C b \sum F^{i\bi} - C b.
\end{aligned}
\end{equation}
Note that
\begin{equation}
\label{hess-a271.5}
\begin{aligned}
 |\nabla_i \phi|^2 
 \leq \,& 2 |\nabla_i  \eta|^2 + C b^2 \fg_{i\bi}^2
               + C b^2 \sum_k |\nabla_i \nabla_{k} |^2 + C b^2
\end{aligned}
\end{equation}
where $C$ depends on $|\nabla u|_{C^0 (\bM)}$.

Finally, combining  \eqref{gblq-C90'''}, \eqref{gblq-R350}, 
\eqref{gblq-R155.4}, 
\eqref{hess-a271.6} and \eqref{hess-a271.5}  we obtain
\begin{equation}
\label{gblq-C151}
\begin{aligned}
\fg_{1\bar{1}} \cL \eta 
\leq \,& A_1 + A_2 + A_3 + 2 b A_4
          \end{aligned}
\end{equation}
if $b$ and $\fg_{1\bar{1}}$ are sufficiently small and large, respectively, where
\begin{equation}
\label{gblq-C1510}
\begin{aligned}
 A_1 = \,& - \nabla_{\bar{1}} \nabla_1  \psi
         - 2 \fRe\{\chi_{1 \bar{1}, \zeta_{\alpha}}  \nabla_{\alpha} \psi\}
        + C b \fg_{1\bar{1}} |\fRe\{\nabla_{\bk} u \nabla_k \psi\}|  \\
 A_2 = \,& 8 \gamma  \fg_{1\bar{1}} F^{i\bi}  |\nabla_i  \eta|^2
        + 2 \fg_{1\bar{1}}  \sum_{J \cup L} F^{i\bi}  |\nabla_i \eta|^2
        + C \fg_{1\bar{1}} \sum F^{i\bi}  \\
 A_3 = \,&  2 F^{i\bi} \fRe\{\chi_{i\bi  \bar{1}, \zeta_{\alpha}}  \nabla_1 \nabla_{\alpha} u\}
      + \frac{C}{\fg_{1\bar{1}}} \sum_{i, \alpha} F^{i\bi} |\chi_{i \bar{1}, \zeta_{\alpha}}
       \nabla_1 \nabla_{\alpha} u|^2
      +Q \\
A_4 =\,& - \frac{ \fg_{1\bar{1}}}{8} F^{i\bi} \fg_{i\bi}^2
                  - \frac{ \fg_{1\bar{1}}}{4} \sum_{i,k} F^{i\bi} |\nabla_i  \nabla_{k} u|^2.
          \end{aligned}
\end{equation}

\subsection{Proof of Theorems~\ref{gj-th4}}
\label{gn-proof-th4}

By Cauchy-Schwarz inequality we obtain
\begin{lemma}
\label{lemma-C2-10}
Under assumption~\eqref{A3},
\begin{equation}
\label{gblq-R1551}
  Q  \leq - \frac{c_0}{8} \Big(\fg_{1\bar{1}}^2
   + \sum_{k} |\nabla_1  \nabla_{k} u|^2\Big) \sum F^{i\bi}
   + C \sum F^{i\bi} |\nabla_1  \nabla_i u|^2 + C \sum F^{i\bi}.
\end{equation}
\end{lemma}

\begin{remark}
If $\chi = \chi (z, u)$ and is independent of  $\p u$, $\bar \p u$, then
\[ F^{i\bi} H_{i\bi} \geq - C \fg_{1\bar{1}} \sum F^{i\bi}. \]
\end{remark}

By Lemma~\ref{lemma-C2-10} and \eqref{gblq-C151}
we derive under assumption~\eqref{A3},
\begin{equation}
\label{gblq-C151'}
\begin{aligned}
\fg_{1\bar{1}} \cL \eta 
     \leq \,&  4  \fg_{1\bar{1}} F^{i\bi}  |\nabla_i  \eta|^2
        + A_1 + b A_4
  - \frac{c_0}{32} \Big(\fg_{1\bar{1}}^2
     + \sum_{k} |\nabla_1  \nabla_{k} u|^2\Big) \sum F^{i\bi}
      \end{aligned}
\end{equation}
if $b$ and $\fg_{1\bar{1}}$ are sufficiently small and large, respectively.

Next, by the concavity of $f$ and Schwarz inequality,
\[ \begin{aligned}
\fg_{1\bar{1}} \sum F^{i\bi}
    = \,& \sum F^{i\bi} (\fg_{1\bar{1}} - \fg_{i\bi})
             + \sum F^{i\bi} \fg_{i\bi} \\
\geq \,& f (\fg_{1\bar{1}} {\bf 1}) - \psi
- \frac{1}{\fg_{1\bar{1}}} \sum F^{i\bi}  \fg_{i\bi}^2
 - \fg_{1\bar{1}}  \sum F^{i\bi}.
\end{aligned} \]
It follows that  when $\fg_{1\bar{1}}$ is sufficiently large,
\begin{equation}
\label{hess-a272.7}
 \fg_{1\bar{1}} \sum F^{i\bi}
\geq c_1  - \frac{1}{2 \fg_{1\bar{1}}} \sum F^{i\bi} \fg_{i\bi}^2
\end{equation}
where
\[ c_1 = \frac{1}{2} \Big(\sup_{\Gamma} f - \sup_{M} \psi\Big) > 0. \]

 Following Guan-Wang~\cite{GW03a} we choose $\eta = \log \rho$ where $\rho$ is a smooth
function with compact support in $B_{R}\subset M$
satisfying
\begin{equation}
\label{2-22}
0\leq \rho \leq 1, ~~\rho|_{B_{\frac{3 R}{4}}}\equiv 1,
~~|\nabla \rho | \leq C_R \sqrt{\rho},
~~|\nabla^{2} \rho| \leq C_R.
\end{equation}
Clearly,
\begin{equation}
\label{2-23}
\begin{aligned}
\mathcal{L} \eta  = \,& \frac{1}{\rho} \mathcal{L} \rho
    - \frac{1}{\rho^2} F^{i\bi} |\nabla_i \rho|^2
   \geq - \frac{C}{\rho} \sum F^{i\bi}.
\end{aligned}
\end{equation}
By \eqref{gblq-C151'}, \eqref{hess-a272.7} and \eqref{2-23},
\begin{equation}
\begin{aligned}
(c_0 \rho \fg_{1\bar{1}} - C) \fg_{1\bar{1}} \sum F^{i\bi}
  + (b \fg_{1\bar{1}} - C) \rho F^{i\bi} \fg_{i\bi}^2
  + (c_0 c_1 - C b) \rho \fg_{1\bar{1}} \leq C
\end{aligned}
\end{equation}
We now fix $b$ sufficiently small to derive a bound
$\rho \fg_{1\bar{1}} \leq C$. 
This yields an interior estimates for $|\partial \bpartial u|$, and a bound for 
 the H\"older norm of $|\partial \bpartial u|$ follows from Evans-Krylov Theorem.
 The proof of Theorem~\ref{gj-th4} is complete.

\subsection{Proof of Theorems~\ref{3I-th4} and \ref{3I-th40}}

As in \cite{Guan14} we take
\[ \eta = A \big(\ul{u} - u + 1 + \sup (u - \ul{u})\big) \]
and denote $\ul {\fg}_{i\bj} = \sqrt{-1} \partial_{\bj} \partial_i \ul u +  \chi_{i\bj} [{\ul u}]$. Then
\begin{equation}
\label{gblq-I65}
 A_2 \leq C (\gamma A^2 + 1) \fg_{1\bar{1}} \sum F^{i\bi}
    + C A^2  \fg_{1\bar{1}} \sum_J F^{i\bi} + CA^2  \fg_{1\bar{1}} F^{1 \bar{1}}.
  \end{equation}
By the concavity of $\chi_{i\bi}$ we have
\[ \begin{aligned}
\chi_{i\bi,\zeta_{\a}}(\ul{u}_{\a}-u_{\a})+\chi_{i\bi, \zeta_{\ba}}(\ul{u}_{\ba}-u_{\ba})
\geq  \chi_{i\bi} (z, \nabla \ul{u}) - \chi_{i\bi} (z, \nabla u).
\end{aligned} \]
Therefore,
\begin{equation}
\label{gblq-I75}
\begin{aligned}
\cL(\ul{u}-u) \geq \,& F^{i\bi} (\ul{u}_{i\bi} - u_{i\bi}) +
                                 F^{i\bi} (\chi_{i\bi} [\ul{u}]) - \chi_{i\bi} [u])
                        = F^{i\bi}(\ul {\fg}_{i\bi} - \fg_{i\bi}).
\end{aligned}
 \end{equation}

\begin{lemma}
\label{gblq-lemma-C20}
There exist uniform constants $\theta, N > 0$
such that either  
\begin{equation}
\label{gblq-C210}
\cL \eta 
\geq \theta \Big(\sum F^{i\bi} + 1\Big)
\end{equation}
or
\begin{equation}
\label{gblq-C215}
F^{1\bar{1}}  \geq \theta \sum F^{i\bi}
\end{equation}
provided that $\fg_{1\bar{1}}  \geq N$.
\end{lemma}

\begin{proof}
Let $\mu = (\mu_1, \ldots, \mu_n)$ and $\lambda = (\lambda_1, \ldots, \lambda_n)$
be the eigenvalues of $\sqrt{-1} \partial \bpartial \ul u + \chi [\ul u]$ and
$ \sqrt{-1} \partial \bpartial u + \chi [u]$, respectively.
  Suppose \eqref{gblq-C215} does not hold. By a lemma in \cite{CNS3} and Theorem~\ref{gn-th10}  we have
 \begin{equation}
 F^{i\bi}(\ul {\fg}_{i\bi} - \fg_{i\bi}) \geq
\sum  f_i (\mu_i - \lambda_i) \geq \theta \Big(\sum F^{i\bi} + 1\Big).
\end{equation}
This proves \eqref{gblq-C210}.
\end{proof}

Suppose $\fg_{1\bar{1}} \geq N$ and \eqref{gblq-C215} holds. Then
\[ A_2 + A_3 \leq C \Big(A^2 \fg_{1\bar{1}} + \fg_{1\bar{1}}^2 + \sum |\nabla_1  \nabla_{k} u|^2\Big) \sum F^{i\bi} \]
and
\[  \begin{aligned}
 A_4 \leq \,& - \frac{\fg_{1\bar{1}}}{8} F^{1\bar{1}} \fg_{1\bar{1}}^2
                     - \frac{ \fg_{1\bar{1}}}{4} \sum F^{1\bar{1}} |\nabla_1  \nabla_{k} u|^2 \\
        \leq \,&  -  \frac{\theta \fg_{1\bar{1}}}{8}
           \Big(\fg_{1\bar{1}}^2 + \sum |\nabla_1  \nabla_{k} u|^2\Big) \sum F^{i\bi}.
   \end{aligned} \]
Moreover, by \eqref{gblq-I75}
\[ \cL \eta \geq - C \fg_{1\bar{1}} \sum F^{i\bi}. \]
Combining these inequalities \eqref{hess-a272.7}, from \eqref{gblq-C151}  we derive
a bound for  $\fg_{1\bar{1}}$.

Suppose now that \eqref{gblq-C210} holds.
To control the term $A_3$ in \eqref{gblq-C151} we need condition \eqref{A5}
which is equivalent to $Q \leq 0$  so that
\begin{equation}
\label{gblq-I63}
A_3 \leq (C + \epsilon \fg_{1\bar{1}}) \sum F^{i\bi} |\nabla_1 \nabla_i u|^2
+ C \fg_{1\bar{1}} \sum F^{i\bi}.
  \end{equation}
We make use of Lemma~\ref{gblq-lemma-C20} to derive from \eqref{gblq-C151}
\begin{equation}
\label{hess-a271.7}
\begin{aligned}
0 \geq \,& (b \fg_{1\bar{1}}^2  - C A^2) F^{1\bar{1}}
    + \sum_J F^{i\bi} (b \fg_{i\bi}^2 - C A^2)
    + \frac{b}{16} F^{1\bar{1}} \fg_{1\bar{1}}^2 \\
   & + (b - \epsilon - C \fg_{1\bar{1}}^{-1}) F^{i\bi} |\nabla_1  \nabla_i u|^2
 + (\theta A - C \gamma A^2 - C) \sum F^{i\bi}.
 \end{aligned}
\end{equation}
Finally, fixing $A$ sufficiently large and $\gamma$, $\epsilon$ sufficiently small,
 we obtain a bound $\fg_{1\bar{1}} \leq CA /\sqrt{b}$, completing the proof ofTheorem~\ref{3I-th4}. Clearly the proof also applies to Theorem~\ref{3I-th40}
 with slight modifications.

\subsection{Proof of Remark~\ref{gn-remark-I10}}

We now refine the above arguments to prove  the claims in
Remark~\ref{gn-remark-I10} when $\psi = \psi (z, u, \nabla u)$.
By straightforward calculations,
\[ \fRe\{\nabla_{\bk} u \nabla_k \psi\}
   = \fRe\{\psi_{\zeta_{\alpha}} \nabla_{\alpha} (|\nabla u|^2\}) + O (1) \]
and therefore
\[ 2 b e^{-h} \fRe\{\nabla_{\bk} u \nabla_k \psi\}
    =  - 2 \fRe\{\psi_{\zeta_{\alpha}} \nabla_{\alpha} h\} + O (1). \]
Let $\widetilde{\cL}$ be defined by
\[ \widetilde{\cL} v = \cL v - 2 \fRe\{\psi_{\zeta_{\alpha}} \nabla_{\alpha} v\}
                      = F^{i\bi} \p_{\bi} \p_i v + 2  \fRe\{(F^{i\bi} \chi_{i \bi, \zeta_{\alpha}} - \psi_{\zeta_{\alpha}}) \nabla_{\alpha} v\}. \]
From \eqref{hess-a271.1} we see that
\begin{equation}
\begin{aligned}
- \widetilde{\cL} h \geq \,&  \frac{b}{2} F^{i\bi} \fg_{i\bi}^2
             + b \sum_{i,k} F^{i\bi} |\nabla_i \nabla_{k} u|^2
               - C b \sum F^{i\bi} - C b.
\end{aligned}
\end{equation}
Next, we calculate
 \begin{equation}
 \begin{aligned}
- \nabla_{\bar{1}} \nabla_1  \psi - \,& 2 \fRe\{\chi_{1 \bar{1}, \zeta_{\alpha}}  \nabla_{\alpha} \psi\} \\
  \leq \,& - 2 \fRe\{\psi_{\zeta_{\alpha}} \nabla_{\alpha}
 \fg_{1 \bar{1}}\}  + C \sum |\nabla_1 \nabla_k u| + C \fg_{1 \bar{1}} + R \\
\leq \,& 2 \fg_{1 \bar{1}} \fRe\{\psi_{\zeta_{\alpha}} \nabla_{\alpha} \phi\}  + C \sum |\nabla_1 \nabla_k u| + C \fg_{1 \bar{1}} - R
 \end{aligned}
 \end{equation}
 where
 \begin{equation}
\label{gblq-R155.7}
 \begin{aligned}
   R   = \,& \psi_{\zeta_{\alpha} \bar{\zeta}_{\beta}}
 (\p_1  \p_{\alpha} u  \p_{\bar{1}} \p_{\bbeta} u + \p_{\bar{1}} \p_{\alpha} u  \p_{{1}} \p_{\bbeta} u)
   + 2 \fRe\{\psi_{\zeta_{\alpha} \zeta_{\beta}}
                 \p_1  \p_{\alpha} u  \p_{\bar{1}} \p_{\beta} u\}.
   \end{aligned}
\end{equation}
So in place of \eqref{hess-a271.6} we derive
\begin{equation}
\begin{aligned}
\fg_{1 \bar{1}} \widetilde{\cL} \phi
   \geq \,&  \fg_{1 \bar{1}} \widetilde{\cL} \eta
             + \frac{b \fg_{1 \bar{1}}}{2} F^{i\bi} \fg_{i\bi}^2
             + b \fg_{1 \bar{1}} \sum_{i,k} F^{i\bi} |\nabla_i \nabla_{k} u|^2
             -  C b \fg_{1 \bar{1}}\sum F^{i\bi} \\
             & -  C \sum |\nabla_1 \nabla_k u| - C \fg_{1 \bar{1}} + R.
\end{aligned}
\end{equation}
Consequently,
\begin{equation}
\begin{aligned}
\fg_{1\bar{1}} \widetilde{\cL} \eta
     \leq \,& 4  \fg_{1\bar{1}} F^{i\bi}  |\nabla_i  \eta|^2  + A_3 + b A_4 + R
      + C \sum |\nabla_1 \nabla_k u| + C \fg_{1 \bar{1}}
      \end{aligned}
\end{equation}
provided that $\fg_{1\bar{1}}$ is sufficiently large.

We have $R \geq 0$ if $\psi$ is convex in $\nabla u$; otherwise
\[ R \geq - C \sum |\nabla_1 \nabla_k u|^2 - C \fg_{1 \bar{1}}^2 \]
and can therefore be controlled by
\[  \Big(\sum |\nabla_1 \nabla_k u|^2 + \fg_{1 \bar{1}}^2\Big) \sum F^{i\bi}  \]
under assumption \eqref{gn-I185}.
In either case, the rest of proof is same as in
Subsection~\ref{gn-proof-th4}.

\bigskip

\section{
The equation from Gauduchon conjecture}
\label{gn-G}
\setcounter{equation}{0}
\medskip

Let $\Omega$ be a closed
real $(1,1)$ form on a compact Hermitian manifold $(M^n, \omega)$
with $[\Omega] = c_1^{BC} (M)$ in the
Bott-Chern cohomology group $H_{BC}^{1,1} (M, \bfR)$. 
Gauduchon~\cite{Gauduchon84} conjectured
that there exists a Gauduchon metric $\hat{\omega}$ on $M$
with Chern-Ricc curvature $\Ric_{\hat{\omega}} = \Omega$.

This is a natural extension of the Calabi conjecture for K\"ahler manifolds
solved by Yau~\cite{Yau78}.
It was discovered by Popovici~\cite{Popovici13c} and
Tosatti-Weinkove~\cite{TWv13b} independently that Gauduchon
conjecture reduces to solving the form-type Monge-Amp\`ere equation~\eqref{CH-I10}-\eqref{CH-I15} with $c = 1$ and $\sup_M u = 0$, where $\omega_0$ is
a Gauduchon metric; see \cite{TWv13b} for details.

When $n = 2$, Equation~\eqref{CH-I10} is the
standard complex Monge-Amp\`ere equation so Gauduchon conjecture
follows affirmatively from results of Cherrier~\cite{Cherrier87}.

Using the Hodge star operator $*$, Tosatti-Weinkove~\cite{TWv13b} converted
equation~\eqref{CH-I10} into a Monge-Amp\`ere type equation for a $(1,1)$-form.
Recall that
\[ \Phi_u =\omega_0^{n-1} + \sqrt{-1} \partial \bpartial u \wedge \omega^{n-2}
+ c \, \fRe \{\sqrt{-1} \partial u \wedge \bpartial \omega^{n-2}\} > 0. \]
Define
\[ \tilde{\omega} =\frac{1}{(n-1)!}*\Phi_u. \]
As in \cite{TWv13b},
\[ \tilde{\omega} = \frac{1}{(n-1)!} * \Phi_u =
\frac{\tilde{\chi} + (\Delta u) \omega - \sqrt{-1} \partial \bpartial u}{n-1} > 0 \]
where
\[ \tilde{\chi} = \frac{1}{(n-2)!} * (\omega_0^{n-1}
+ c \, \fRe\{\sqrt{-1} \partial u \wedge \bpartial \omega^{n-2}\}), \]
and equation~\eqref{CH-I10}
becomes
\begin{equation}
\label{gn-G30}
 \tilde{\omega}^n = e^{h+b}\omega^n.
\end{equation}

It was shown by Tosatti-Weinkove~\cite{TWv13b} that
the classical solvability of equation~\eqref{gn-G30} reduces to the second order estimate
\begin{equation}
\label{hess-a10d}
 \Delta u \leq C \Big(1 + \sup_M |\nabla u|^2\Big);
\end{equation}
see Conjecture 1.5 in \cite{TWv13b}.
Later on, Szekelyhidi-Tosatti-Weinkove~\cite{STW17} derived such an estimate and consequently proved Gauduchon conjecture.

Let
\[ \chi = \frac{\tr_{\omega} \tilde{\chi}}{n-1} \omega - \tilde{\chi} \]
 so
$\tilde \chi = (\tr_{\omega} {\chi}) \omega - \chi$.
Equation~\eqref{gn-G30} can therefore be rewritten in the form
\begin{equation}
\label{gn-G40}
\begin{aligned}
\log \rho_{n-1} (\lambda (\sqrt{-1} \partial \bpartial u + \chi)) = \psi (z).
\end{aligned}
\end{equation}
In the rest of this section we verify that Theorem~\ref{3I-th4} applies to equation~\eqref{gn-G40}.

First of all, note that $\chi [u]$ is linear in $\nabla u$ so \eqref{A2} is satisfied.
We may take $\ul u = 0$ as $\chi_0 > 0$ and
$\widetilde{\cC}_{\sigma}^+ = \cP_{n-1}$ for $f = \log \rho_{n-1}$.
It remains to verify \eqref{A5} for $\alpha = 1$ since $r_0 = n-1$ by
Lemma~\ref{gn-lemma-R20}.

In local coordinates computing at a point where $g_{k\bl}=\delta_{kl}$ and
$T_{ij}^k = 2\Gamma_{ij}^k$,  we have
\begin{equation}
\begin{aligned}
  \sqrt{-1} \partial u \wedge \bpartial \omega^{n-2}
 = & - (n-2)    \partial_p u \partial_{\bq} g_{k\bl}
    dz_p \wedge d\bz_q \wedge dz_k \wedge d\bz_l \wedge \omega^{n-3} \\
 = & \frac{n-2}{2} \Big(\sum_{ i} \sum_{p,l \neq i} \partial_p u \ol{T_{pl}^{l}} \mu_{i\bi}
     - 
      \sum_{i \neq j} \sum_{l \neq j}
            (\partial_j u \ol{T^l_{il}} + \partial_l u \ol{T^j_{li}}) \mu_{i\bj}\Big)
      \end{aligned}
\end{equation}
where $\mu_{i\bj}  = s_{ij} dz_1 \wedge \cdots \wedge \widehat{d z_i} \wedge d \bz_{i} \wedge\cdots \wedge dz_j \wedge \widehat{d\bz_{j}} \wedge \cdots \wedge dz_n \wedge d \bz_{n}$ and
\[ s_{ij} = \left\{ \begin{aligned} (\sqrt{-1})^{n-1}, \;\, i \leq  j, \\ -(\sqrt{-1})^{n-1}, \;\, i > j. \end{aligned} \right. \]
Similarly,
\begin{equation}
\begin{aligned}
   - \sqrt{-1} \bpartial u \wedge \partial \omega^{n-2}
 = & -  (n-2)  \partial_{\bq} u \partial_{p} g_{k\bl}
    dz_p \wedge d\bz_q \wedge dz_k \wedge d\bz_l \wedge \omega^{n-3} \\
 = & \frac{n-2}{2} \Big(\sum_{ i} \sum_{p,l \neq i} \partial_{\bp} u {T_{pl}^{l}} \mu_{i\bi}
     - \sum_{i \neq j} \sum_{l \neq i}
            (\partial_{\bi} u {T^l_{jl}} + \partial_{\bl} u {T^i_{lj}}) \mu_{i\bj}\Big).
      \end{aligned}
\end{equation}
It follows that (see \cite{TWv13b})
\[ \f{1}{(n-2)!} * \fRe\{\sqrt{-1} \partial u \wedge \bpartial \omega^{n-2}\}
=\sqrt{-1} \widetilde{E}_{i\bj} dz_i \wedge d \bz_j\]
where
\begin{equation}
\label{gn-G50}
 \widetilde{E}_{i\bi} = \frac{1}{2}
    \sum_{p, l \neq i} (\partial_p u \ol{T_{pl}^{l}} + \partial_{\bp} u {T_{pl}^{l}}),
\end{equation}
and
\[   \widetilde{E}_{i\bj} = - \frac{1}{2}
   \Big(\sum_{l \neq i} (\partial_i u \ol{T^l_{jl}} + \partial_l u \ol{T^i_{lj}})
   + \sum_{l \neq j} (\partial_{\bj} u {T^l_{il}} + \partial_{\bl} u {T^j_{li}})\Big), \;\; i \neq j. \]
Clearly $\widetilde{E}_{i\bj}$ does not contain $\partial_j u$. Therefore,
as $\tilde{\chi}_{i\bj \bk} = \tilde{\chi}_{i\bj,  \bk} - \ol{\Gamma_{kj}^l} \tilde{\chi}_{i\bl}$ and $\Gamma_{kk}^l = 0$,
\begin{equation}
\label{gn-G60}
\tilde{\chi}_{i\bj,\zeta_j} = c \, \widetilde{E}_{i\bj, \zeta_{j}} = 0,
\;\;  \tilde{\chi}_{i\bi \bi, \zeta_i} = \tilde{\chi}_{i\bi, \zeta_i\bi} = c \, \widetilde{E}_{i\bi, \zeta_i\bi} = 0, \;
\forall \, i, j.
\end{equation}


Consider
\[ F (\sqrt{-1} \partial \bpartial u + \chi)
      = \log \rho_{n-1} (\lambda (\sqrt{-1} \partial \bpartial u + \chi)). \]
Denote $\lambda_i = \fg_{i\bi}$ and
$\eta_j = \lambda_1 + \cdots
   + \widehat{\lambda}_j + \cdots +\lambda_n$ at $p$.
Then
\begin{equation}
\label{gn-G70}
F^{i\bi} = \sum_{j \neq i} \frac{1}{\eta_j}.
  \end{equation}
As $\lambda_1 \geq \cdots \geq \lambda_n$, we have  $
F^{1\bar{1}} \leq \cdots \leq F^{n\bn}$ and
\[ F^{k\bk} \geq \frac{1}{\eta_1} 
   \geq \frac{1}{n-1}F^{k\bk}, \;\; \forall \, k \geq 2. \]
Next,
\begin{equation}
\label{gn-G80}
\begin{aligned}
\sum F^{i\bi} \chi_{i\bi \bar{1}, \zeta_1}
   = \,& \sum_i \sum_{j \neq i} \frac{\chi_{i\bi \bar{1}, \zeta_1}}{\eta_j}
          = \sum_j \sum_{i \neq j} \frac{\chi_{i\bi \bar{1},\zeta_1}}{\eta_j} \\
   = \,& 
       \sum_{j} \frac{ \tilde{\chi}_{j\bj \bar{1},\zeta_1}}{\eta_j}
   =   \sum_{j \neq 1} \frac{ \tilde{\chi}_{j\bj \bar{1},\zeta_1}}{\eta_j}
\end{aligned}
\end{equation}
since $\tilde{\chi}_{1\bar{1} \bar{1},\zeta_1}
    = c \, \widetilde{E}_{1\bar{1} \bar{1},\zeta_1} = 0$ by \eqref{gn-G60}.
Finally, from \eqref{gn-G70} and \eqref{gn-G80} we derive
\begin{equation}
\label{gn-G90}
|\sum F^{i\bi} \chi_{i\bi \bar{1}, \zeta_1}|  \leq C \sum_{j \neq 1} \frac{1}{\eta_j}
   = C F^{1\bar{1}}.
\end{equation}
Thus \eqref{A5} is verified.

\end{document}